\documentclass[aos, preprint]{imsart}

\RequirePackage{amsthm,amsmath,amsfonts,amssymb,mleftright}
\RequirePackage[numbers]{natbib}
\usepackage{bbm}
\usepackage{xparse}
\usepackage{mathrsfs}
\usepackage{enumitem}

\RequirePackage[colorlinks,citecolor=blue,urlcolor=blue]{hyperref}%% uncomment this for coloring bibliography citations and linked URLs

\startlocaldefs
\theoremstyle{plain}
\newtheorem{thm}{Theorem}
\newtheorem{lemma}{Lemma}
\newtheorem{prop}{Proposition}
\newtheorem{cor}{Corollary}
\theoremstyle{remark}
\newtheorem{remark}{Remark}

\newcommand{\K}{\mathcal{K}}
\newcommand{\floor}[1]{\left \lfloor #1 \right \rfloor} %floor
\newcommand{\X}{\mathcal{X}}
\newcommand{\Q}{\mathcal{Q}}
\newcommand{\T}{\mathbb{T}}
\newcommand{\EE}{\mathbb{E}}
\newcommand{\RR}{\mathbb{R}}
\newcommand{\eps}{\varepsilon}
\newcommand{\lle}{\lesssim}
\newcommand{\SSS}{\mathcal{S}}
\newcommand{\eL}{\mathscr{L}}
\newcommand{\Nzero}{\mathbb{N}_0}
\newcommand{\del}{\partial}
\newcommand{\ip}[2]{\bigl\langle #1, #2 \bigr\rangle} %inner product
\newcommand{\Bigip}[2]{\Bigl\langle #1, #2 \Bigr\rangle} %inner product
\newcommand{\II}{\mathcal{I}}
\newcommand{\JJ}{\mathcal{J}}
\NewDocumentCommand{\evalat}{sO{\big}mm}{%
	\IfBooleanTF{#1}
	{\mleft. #3 \mright|_{#4}}
	{#3#2|_{#4}}%
}

\newcommand{\HHH}{\mathbb{H}}
\newcommand{\gaatp}{\overset{P}{\to}}

\endlocaldefs

\begin{document}

\begin{frontmatter}
\title{Semi-parametric Bernstein--von Mises Theorem in a Parabolic PDE Problem \\ \\ \rm{ Dedicated to the memory of Harry van Zanten} }

\runtitle{Semi-parametric BvM in a Parabolic PDE Problem}
%\thankstext{T1}{A sample additional note to the title.}

\begin{aug}
	%%%%%%%%%%%%%%%%%%%%%%%%%%%%%%%%%%%%%%%%%%%%%%%
	%% Only one address is permitted per author. %%
	%% Only division, organization and e-mail is %%
	%% included in the address.                  %%
	%% Additional information can be included in %%
	\author[A]{\fnms{Adel}~\snm{Magra} \ead[label=e1]{a.magra@vu.nl}}
	\author[A]{\fnms{Frank}~\snm{van der Meulen}\ead[label=e2]{f.h.van.der.meulen@vu.nl}}
	\and
	\author[B]{\fnms{Aad}~\snm{van der Vaart}\ead[label=e3]{A.W.vanderVaart@tudelft.nl}}
	%%%%%%%%%%%%%%%%%%%%%%%%%%%%%%%%%%%%%%%%%%%%%%
	%% Addresses                                %%
	%%%%%%%%%%%%%%%%%%%%%%%%%%%%%%%%%%%%%%%%%%%%%%
	\address[A]{
		VU Amsterdam \printead[presep={ ,\ }]{e1,e2}}
	
	\address[B]{
		TU Delft \printead[presep={,\ }]{e3}}
\end{aug}

\begin{abstract}
	We consider the heat equation with absorption in a bounded domain of $\mathbb{R}^d$, where both the scalar diffusivity and the absorption function are unknown. We investigate a Bayesian approach for recovering the diffusivity from a noisy observation of the solution to the PDE over the domain. Given a Gaussian process prior on the absorption function, we derive a Bernstein-von Mises theorem for the marginal posterior distribution of the diffusivity under assumptions on the prior and on smoothness properties of the absorption. 
\end{abstract}

\begin{keyword}[class=MSC]
\kwd[Primary ]{62F15}
\kwd[; secondary ]{35R30}
\end{keyword}

\begin{keyword}
\kwd{nonlinear inverse problem}
\kwd{Bayesian}
\kwd{uncertainty quantification}
\end{keyword}

\end{frontmatter}
%%%%%%%%%%%%%%%%%%%%%%%%%%%%%%%%%%%%%%%%%%%%%%
%% Please use \tableofcontents for articles %%
%% with 50 pages and more                   %%
%%%%%%%%%%%%%%%%%%%%%%%%%%%%%%%%%%%%%%%%%%%%%%
%\tableofcontents

%%%%%%%%%%%%%%%%%%%%%%%%%%%%%%%%%%%%%%%%%%%%%%
%%%% Main text entry area:

\section{Introduction}

Let $\X$  be a bounded subset of $\mathbb{R}^d$, with boundary $\del \X$, and let $T<\infty$ be a fixed time horizon. 
For a given scalar \textit{diffusivity} $\theta>0$ and \textit{absorption function}
$f:\X \to (0,\infty)$, consider the following parabolic partial differential equation (PDE):
\begin{equation}\label{eq:PDE}
	\left\{\begin{aligned}
		\del_t u - \frac{1}{2}\theta\Delta_x u &=- fu,\quad&& \text{on }\X\times (0,T),\\
		u &= g, &&\text{on }\del\X\times (0,T),\\
		u(\cdot,0) &= u_0, &&\text{on }\X.
	\end{aligned}\right.
\end{equation}
Here $\del_t $ is the partial derivative relative to $t$ of the function $(x,t)\mapsto u(x,t)$ and $\Delta_x = \sum_{i=1}^d \del^2/\del x_i^2$ is the Laplace operator, and 
$u_0$ and $g$ are known boundary conditions. Under suitable assumptions (see the next section), 
equation  \eqref{eq:PDE} possesses a classical unique solution $u: \X\times (0,T)\to\mathbb{R}$. Denoting this solution, which depends on $(\theta,f)$, 
by $\K_\theta f = u_{\theta,f}$, we are interested in the problem of recovering the pair $(\theta,f)$ from noisy observations of $\K_\theta f$. 
We are particularly interested in uncertainty quantification for $\theta$, and
thus complement the paper \cite{kekkonen}, who considered the recovery of $f$ for a known value $\theta=1$.

Assuming $\K_\theta f$ to belong to $L^2(\X\times (0,T))$, we consider the signal-in-white-noise observational model, where the observation is 
a corrupted solution $\K_\theta f$ over the whole domain $\X\times [0,T]$. Namely, for a signal to noise ratio $n\in\mathbb{N}$, we observe
\begin{align}\label{eq: observation}
	\K_\theta f + \frac{1}{\sqrt{n}}\dot{W},
\end{align}
where $\dot{W}$ is the iso-Gaussian process on $L^2(\X\times (0,T))$, 
that is, $\dot{W}$ is a stochastic process $\{\dot{W}_h : h\in L^2(\X\times(0,T))\}$ satisfying $\dot{W}_h\sim N(0,\|h\|_{L^2(\X\times (0,T))}^2)$.
As is well known, this is asymptotically equivalent in the Le Cam sense to the regression model in which one observes
the function $\K_\theta f$ at $n$ suitably placed ``design points'' in $\X\times (0,T)$ with i.i.d.\ standard normal errors.

We take a Bayesian approach with $\theta$ and $f$ independent under the prior. The prior on $\theta$ only needs to have positive Lebesgue density, 
while we choose a Gaussian process prior for $f$. We investigate the resulting posterior distribution of $(\theta, f)$ under the assumption that in reality 
the data in \eqref{eq: observation} is generated according to a ``true'' pair of parameters $(\theta_0,f_0)$. Our main novel result is a 
Bernstein-von Mises (BvM) theorem for the marginal posterior of $\theta$.

Our setting is a statistical inverse problem, as the accessible observation concerns the transformation $\K_\theta f$ of the object of interest $(f,\theta)$ rather than the latter object itself. Nonparametric Bayes methods  for inverse problem have become popular in recent years for their convenience and practicality for computations, and their asymptotic performance has been substantially studied. Inverse problems can be categorised as linear (e.g. \ \cite{Knapiketal2011,GiordanoKekkonen, Knapiketal2013,Ray2013,Yan2020})  or nonlinear (e.g. \ \cite{Nickl2018, Nickl23, kekkonen, MonardNicklPternain2021, monard2021,giordano,nicklVdG,SzabovdV2024}) based on the linearity of the transformation involved. Both types of inverse problems present their challenges, with nonlinear problems making it particularly difficult to come up with a general theory from which to directly extract results. The important reference \cite{Nickl23} provides a general scheme to deal with nonlinear inverse problems when using Gaussian priors, but specifics must be investigated one problem at a time. 

Statistical inference on  $f$ in the problem \eqref{eq:PDE} with $\theta$ set equal to 1 is considered In \cite{kekkonen}. Various contraction results are derived. 
The focus in the present paper is on the parameter $\theta$, when both $f$ and $\theta$ are unknown, making this a semi-parametric problem. 
Semi-parametric problems are a particular instance of nonparametric problems where the goal is inference on a finite-dimensional function of the unknown parameter. Bayesian methods to deal with such problems have been studied also in the context of inverse problems, usually with a focus on the estimation of a linear functional of the unknown parameter (e.g.\ \cite{Knapiketal2011, ismaeljudith, RayvdV, Nickl2018, GiordanoKekkonen, MonardNicklPternain2021}). The particular semi-parametric structure we consider here can be described as ``strictly semiparametric'', in the sense that the finite-dimensional parameter of interest $\theta$ varies independently from the infinite-dimensional nuisance parameter $f$. In our setting one may see $\theta$ as capturing knowledge of the ``nonlinear operator'' $\K_\theta$, and $f$ as the traditional parameter of the problem of inverting $\K_{\theta}f$. Such a setup has been considered in the literature for different operators $\K_\theta$. For example in \cite{castillo} (section 2), the case of translation over symmetric functions is investigated. In \cite{magra}, general compact linear operators are considered and BvM results are derived when a noisy sample of $f$ is available in addition to \eqref{eq: observation}, with as examples convolution operators and the solution map of the heat equation. %Although here we do not have access to this extra sample, the techniques used to derive a BvM for $\theta$ will be similar, with added difficulty stemming from the nonlinearity of $\K_\theta$.
 
Our main result is a BvM theorem and is presented as Theorem~\ref{thm: main}. It states that for a sufficiently regular true function $f_0$ and a suitable Gaussian prior on $f$, the marginal posterior distribution of $\theta$ behaves asymptotically like a normal distribution centered around an efficient estimator of $\theta$, with variance the inverse of the efficient Fisher information. The semiparametric background isd
developed in Section~\ref{SectionLAN}.

The proof of Theorem~\ref{thm: main} is built on two foundations. The first is contraction of the full posterior to the true parameter pair $(\theta_0, f_0)$ at a polynomial rate in $n$ (Proposition~\ref{prop: contract2}). The second is a controlled locally asymptotically normal (LAN) expansion of the log-likelihood. The contraction is essentially obtained in big steps, following the framework laid down in \cite{Nickl23}. We begin by establishing contraction in the ``direct problem'', i.e.\ of the posterior distribution of $\K_{\theta}f$  at $\K_{\theta_0}f_0$ (Proposition~\ref{prop: contract1}). The main ingredient for this result is a Lipschitz estimate on the map $\K_\theta f$ presented in Lemma~\ref{lemma: lipschitz}. The second step is to ``invert'' this contraction by a stability estimate (Lemma~\ref{lemma: stability}). Finally, controlling the remainder in the LAN expansion is achieved thanks to sufficient regularity of $\K_\theta f$ in both $f$ and $\theta$ (Lemmas~\ref{lemma: bounds} and~\ref{lemma: refinedBounds}). These regularity estimates are themselves derived from PDE theory.

In Section~\ref{sec: prelims}, we set the stage for our statistical analysis. We present our main results in Section~\ref{sec: main}. Their proofs are given in Sections~\ref{sec:proofBvM}, \ref{sec:proofContraction} and \ref{sec:proofInfo}. The various PDE results used throughout can be found in Section~\ref{sec:PDE}.

\subsection{Notation}

For two numbers $a$ and $b$, we denote by $a\wedge b$ their minimum and by  $a\vee b$ their maximum, and
write $a\lesssim b$ if $a\le Cb$ for a universal number $C$.
For two sequences $a_n$ and $b_n$, we mean by $a_n\asymp b_n$ that $|a_n/b_n|$ is bounded away from zero and infinity
as $n\to\infty$, and by $a_n \lle b_n$ that $a_n/b_n$ is bounded above.
%%%%%%%%%%%%%%%%%%%%%%%%%%%%%
For a positive real number $x$, its floor $\floor{x}$ is the largest integer less than or equal to $x$.
%%%%%%%%%%%%%%%%%%%%%%%%%%%%%
We set $\Nzero=\mathbb{N}\cup \{0\}$.  For a multi-index $j\in\Nzero^d$, we define $|j|=\sum_{i=1}^d j_i$.
%%%%%%%%%%%%%%%%%%%%%%%%%%%%%
We write $\del A$ for the boundary of a set $A$ in a metric space and $\overline{A}=A\cup \del A$ for its closure.
For  $\eps > 0$, we denote by  $N(\eps, A, d)$ the minimum number of balls of radius $\eps$ relative to the metric $d$ needed to cover $A$. 

For a Hilbert space $X$, we denote by $X^*$ its topological dual, and for a linear operator $A:X\to Y$ between two Hilbert spaces, by $A^*$ the adjoint of $A$.

%%%%%%%%%%%%%%%%%%%%%%%%%%%%%
For an open or closed subset $\X\subset \mathbb{R}^d$, the notations $C(\X)=C^0(\X)$, $C^r(\X)$, $C^\infty(\X)$, $L^2(\X)=H^0(\X)$ and $H^r(\X)$ have their usual meaning
as the spaces of continuous functions, H\"older functions of degree $r\ge 0$, infinitely often differentiable functions, square-integrable functions and
Sobolev functions of degree $r\ge 0$. For $r<0, \ H^{-r}(\X):= (H^r(\X))^*$. We use abbreviations $\|\cdot\|_\infty$ for the supremum norm, 
and $\ip{f}{g}_2=\int_\X f(x)g(x)\,dx$ and $\|f\|_2:=|\ip{f}{f}_2|^{1/2}$ for the inner product and norm of $L^2(\X)$. The norms of the other
spaces are written with the space as subscript, e.g.\  $\|\cdot\|_{C^r(\X)}$, where the set $\X$ may be deleted if is clear from the context.
%%%%%%%%%%%%%%%%%%%%%%%%%%%%%

In connection with the PDE \eqref{eq:PDE} we also need the parabolic extensions $C^{r,s}\bigl(\X\times(0,T)\bigr)$ and $H^{r,s}\bigl(\X\times(0,T)\bigr)$ of the H\"older
and Sobolev spaces of degree $(r,s)\in [0,\infty)^2$, which roughly are functions $(x,t)\mapsto f(x,t)$ that are smooth of order $r$ in $x$ and of order $s$ in $t$. (In our case always with $s=r/2$.)
For a precise definition of $C^{r,s}\bigl(\X\times(0,T)\bigr)$, see for instance Section~2.1 in \cite{kekkonen}, and for $H^{r,s}\bigl(\X\times(0,T)\bigr)$, see 
\cite{lionsVol2}, Section~2.1 in Chapter~4. In particular, for $r,s\in \mathbb{N}_0^2$, the square of the norm of $H^{r,s}\bigl(\X\times(0,T)\bigr)$ is given by
\begin{align*}
  \|f\|^2_{H^{k,m}(\X\times(0,T))} = \int_0^T \|f(\cdot,t)\|_{H^k(\X)}^2 \,dt + \int_\X \|f(x,\cdot)\|^2_{H^m((0,T))}\,dx.
\end{align*}
Note that this does not include mixed derivatives between the space and time variables. These parabolic spaces are the classical solution
spaces of evolution equations of type \eqref{eq:PDE}. We shall not need the precise definition, but will repeatedly use the inequalities:
\begin{align}
\|fg\|_{H^{r,r/2}}&\lesssim\|f\|_{H^{r,r/2}}\|g\|_{H^{r,r/2}},\qquad r>d/2,\label{EqSobolevProductOne} \\
\|fg\|_{H^{r,r/2}}&\lesssim \|f\|_{C^{r,r/2}}\|g\|_{H^{r,r/2}},\qquad r\ge 0, \label{EqSobolevProductTwo}\\
\|f\|_{H^{(1-\nu)r,(1-\nu) r/2}}&\lesssim \|f\|_2^{\nu}\|f\|_{H^{r,r/2}}^{1-\nu},\qquad \nu\in(0,1),r\ge 0.\label{EqInterpolation}
\end{align}
(See \cite{lionsVol2}, Chapter~4, Proposition~2.1 for the third one.)
The space $H_0^{r,s}\bigl(\X\times(0,T)\bigr)$ is the closure in  $H^{r,s}\bigl(\X\times(0,T)\bigr)$ of the subset of  functions that vanish on a neighbourhood of the
``parabolic boundary'' $\del\X\cup(\X\times\{0\})$, while
$H_{B,0}^{2,1}\bigl(\X\times(0,T)\bigr)$ and $H_{C,0}^{2,1}\bigl(\X\times(0,T)\bigr)$ are the functions in $H^{2,1}\bigl(\X\times(0,T)\bigr)$ 
that vanish on  $\del\X\cup(\X\times\{0\})$ and $\del\X\cup(\X\times\{T\})$, respectively.

\section{A Statistical Inverse Problem: The Heat Equation with Absorption}\label{sec: prelims}
Assume that $\X$ is a bounded and open subset of $\mathbb{R}^d$, for some $d\in\mathbb{N}$, 
with smooth boundary $\del \X$.
For a fixed time horizon $T>0$, let $Q=\X\times (0,T)$ be the space-time cylinder.
We write $\Sigma=\del \X\times (0,T)$ for its ``lateral boundary'',  and refer to $\Sigma\cup (\overline{\X}\times\{0\})$ as its ``parabolic boundary''.

\subsection{Solution to the Classical Inverse Problem} \label{sec: classicalSol}

%%%%%%%%%%%%%%%%%%%%%%%%%%%%%

%\color{blue} %might not be needed 
%Let $0< \eta\leq 1$ and suppose that in \eqref{eq:PDE}, we have that $\theta>0,\ f\in C^\eta(\X),\ g\in C^{2+\eta,1+\eta/2}(\overline{\del \X\times (0,T)}),$ and  $u_0\in C^{2+\eta}(\overline{\X})$. If $f\geq f_{min}>0$, and if the following boundary conditions hold for $x\in\del X$:
%\begin{align}\label{eq:boundary}
	%g(x,0)=u_0(x),\;\;\; \del_tg(x,0)-\frac{\theta}{2}\Delta_x u_0(x) + f(x)u_0(x) = 0,
%\end{align}
%then it is known that a unique and positive solution $u_{\theta, f}$ to \eqref{eq:PDE} exists and belongs to $C(\overline{\Q})\cap C^{2+\eta, 1+\eta/2}(\Q)$. 

For a given pair $(\theta,f)\in(0,\infty)\times L^2(\X)$, consider the parabolic operator  
$	\eL_{\theta,f}:H^{2,1}(\Q)\to L^2(\Q)$, given by 
\begin{align}\label{eq:Ldef}
\eL_{\theta,f}u = \frac{\theta}{2}\Delta_x u -\del_t u - fu.
\end{align}
With this notation the PDE problem \eqref{eq:PDE} can be equivalently rewritten as
\begin{equation}\label{eq:PDE2}
	\left\{\begin{aligned}
		\eL_{\theta,f}u &= 0,\quad&& \text{on }\X\times (0,T),\\
		u &= g, && \text{on }\del\X\times (0,T),\\
		u(\cdot,0)& = u_0, && \text{on }\X.
	\end{aligned}\right.
\end{equation}
For existence of a solution $u$ of a given regularity to this problem, the function $f$ and the boundary functions $g$ and $u_0$ must be ``compatible''. 
Fix $\beta>d/2$, and assume that $f\in C^\infty(\X)$, $g\in H^{3/2+\beta,3/4+\beta/2}(\Sigma)$ and $u_0\in H^{1+\beta}(\X)$ are given functions
such that there exists $\psi\in H^{2+\beta,1+\beta/2}(\overline{\Q}) $ with 
\begin{equation}\label{eq:CR}
	\left\{\begin{aligned}
		\psi &= g, \quad&& \text{ on } \ \del \X \times (0,T),\\
		\psi(\cdot,0) &= u_0, && \text{ on } \ \X,\\
		\evalat{\del^k_t \left(\eL_{\theta,f} \psi \right)}{t=0} &= 0, && \text{ for }\ 0\leq k < \frac{\beta}{2}-\frac{1}{2}.
	\end{aligned}\right.
\end{equation}
(For $\beta<1$, the third condition is empty.)
Then there exists a classical unique solution $u_{\theta,f}\in H^{2+\beta,1+\beta/2}(\Q)$ to \eqref{eq:PDE2} (\cite{lionsVol2}, Theorem~5.3 in Chapter~4,
\cite{Lunardi1995}, \cite{krylov}). In the case that $\X$ is the $d$-dimensional torus $\T^d:=[0,1]^d$ (should be understood in a periodic sense; identifying the right side with the left and the top one with the bottom one), the existence of a (periodic) solution $u_{\theta,f}$ also follows. The fact that $\del \T^d =\varnothing$ means that the existence of a ``compatibility" function $\psi\in H^{3/2+\beta, 3/4+\beta/2}(\Q)$ satisfying \eqref{eq:CR} can be automatically obtained by a trace extension argument (e.g. Theorem 4.2 in Chapter 1 of \cite{lionsVol1}). A consequence of this is that the condition that $f$ is smooth can also be relaxed as per the following theorem, whose proof can be found in Section \ref{sec:PDE}.

\begin{thm}[Existence of Solution]\label{thm: existence}
	Let $\X=\T^d$. Assume $\beta>2+d/2$ and suppose that $f\in H^\beta(\X)$ with $f\geq f_{\min}>0$. Assume also that $u_0\in H^{1+\beta}(\X)$. Then the following boundary value problem
	\begin{equation}\label{eq:PDETorus}
		\left\{\begin{aligned}
			\eL_{\theta,f}u &= 0,\quad&& \text{on }\X\times (0,T],\\
			u(\cdot,0)& = u_0, && \text{on }\X,
		\end{aligned}\right.
	\end{equation}
	has a unique strong solution $u_{\theta,f}\in H^{2+\beta, 1+\beta/2}(\Q)$. Furthermore, there exists $C>0$ such that we have the following estimate:
	\begin{align}\label{eq: estimateSolution}
		\|u\|_{H^{2+\beta, 1+\beta/2}} \leq C(1+\|f\|_{H^{\beta}}^{1+\beta/2}).
	\end{align}
\end{thm}
 
Our problem is to infer the pair $(\theta, f)$ from a noisy observation of $u_{\theta,f}$ and known values of the boundary functions $g$ and $u_0$.
Although the operator $\eL_{\theta,f}$ is linear in $u$ (and also in $(\theta,f)$), the map $(\theta,f)\mapsto u_{\theta,f}$ is non-linear,
and hence this problem is referred to as a non-linear inverse problem.

Assume that $f$, $g$ and $u_0$ are all lower bounded by positive constants. Then it can be shown through the Feynman-Kac representation of $u_{\theta,f}$ that $u_{\theta,f} > 0$ and that $\|u_{\theta,f}\|_\infty\leq \|u_0\|_\infty + \|g\|_\infty$ (\cite{kekkonen}, formula (6)).
The equality $\eL_{\theta,f}u_{\theta,f} = 0$ then yields  $f = \bigl((\frac{\theta}{2}\Delta_x - \del_t)u_{\theta,f}\bigr)/u_{\theta,f}$, showing that
$f$ can be recovered from $u_{\theta,f}$ for known $\theta$.

If $\theta$ is also unknown, then the map $(\theta,f)\mapsto u_{\theta,f}$ is not necessarily injective. In particular, this means that identifying $\theta$ from $u_{\theta,f}$ is not possible in general. Thankfully, the next lemma shows that it is possible to identify the pair $(\theta,f)$ in some cases. 
\begin{lemma}\label{lemma: identifiability}
	Assume that $f, g$ and $u_0$ are all lower bounded by positive constants and that $u=u_{\theta,f}$ satisfies the boundary value problem \eqref{eq:PDE2}. If $\del_t^2\log{u}\neq 0$, then for any $u'= u_{\theta',f'}$ such that $\eL_{\theta', f'} u'=0$, we have
	\[u'=u \implies (\theta,f) = (\theta', f').\]
\end{lemma}
\begin{proof}
	We begin with the observation that the assumption $\del^2_t\log{u}\neq 0$ means that $\del_t \log{u}$ is not constant with respect to time. We have seen in the preceding paragraph that under the assumptions of the lemma, $u$ is positive over $\Q$. Since $u=u'$, we can therefore write:
	\[ f = \frac{\frac{\theta}{2}\Delta_x u-\del_t u}{u}, \;\;\;  f' = \frac{\frac{\theta'}{2}\Delta_x u-\del_t u}{u},\]
	which implies that 
	\[2(f-f') = (\theta-\theta')\frac{\Delta_xu}{u}.\]
	Note that the left hand side of the above equation is constant with respect to time. On the other hand, if $\theta\neq\theta'$, then the right hand side is non-zero and depends non-trivially on time by our first observation since
	\[\frac{\Delta_x u}{u} = \frac{2}{\theta}\left(f + \frac{\del_tu}{u}\right) = \frac{2}{\theta}\left(f + \del_t\log{u}\right).\]
	It follows that $\theta=\theta'$ and consequently $f=f'$.
\end{proof}

\subsection{Parameter Space and Link Function}

To allow for a more streamlined exposition of our results, we will be working on the $d$-dimensional torus $\X=\T^d$ for the rest of the paper. We note that the case of a general bounded subset of $\RR^d$ with smooth boundary $\del\X$ can be dealt with by using the same strategy we employ. Most of our results can actually be adapted to hold in this general case but require to be more careful in dealing with the non-trivial boundary $\del \X$.  We assume that $\theta$ belongs to a compact set $\Theta\subset  [\theta_{\min},\theta_{\max}]\subset (0,\infty)$. For a given integer $\beta> 2+d/2$ and $f_{\min} >0$, we consider the following parameter space for $f$:
\begin{align*}
	\mathcal{F}_{\beta,f_{\min}} = \left\{ f\in H^\beta(\X): \inf_{x\in\X} f(x)>f_{\min}\right\}.
\end{align*}
%We put the same assumptions on the smoothness of $g$ and $u_0$ as presented in the previous subsection, as well as the compatibility relations \eqref{eq:CR}. We also add that $u_0\geq u_{\min}>0$ and $g\geq g_{\min} >0$.
We assume further that $u_0\geq u_{0,\min}>0$. Because $\mathcal{F}_{\beta,f_{\min}}$ is not a linear space and we wish to use a Gaussian prior, we reparametrise using a link function,
 as done in other instances in the literature (\cite{kekkonen, giordano, nicklVdG}). We fix a link function 
$\Phi:\mathbb{R}\to(f_{\min},\infty)$ with bounded derivatives of any order and strictly positive first derivative.
The map $F\mapsto \Phi\circ F$ is a bijection between $H^\beta(\X)$ and $\mathcal{F}_{\beta, f_{\min}}$ and hence
\begin{align*}
	\mathcal{F}_{\beta, f_{\min}} = \{\Phi\circ F: F\in H^\beta(\X)\}.
\end{align*} 
Instead of considering the map $(\theta,f)\mapsto \K_{\theta} f = u_{\theta,f}$ as mentioned in \eqref{eq: observation} (which is well defined by Theorem \ref{thm: existence}), we then define the forward map
as $K_\theta: H^\beta(\X)\to H^{2+\beta,1+\beta/2}(\Q)$ by
\begin{align*}
K_\theta(F) &= \K_\theta(\Phi\circ F) = u_{\theta,\Phi\circ F}.
\end{align*}
By an abuse of notation, we refer to $u_{\theta,\Phi\circ F}$ as $u_{\theta,F}$ and to $\eL_{\theta,\Phi\circ F}$ as $\eL_{\theta,F}$.
We also abbreviate $H:=H^\beta(\X)$.

\subsection{Observational Model}
For an arbitrary orthonormal basis $\{e_k\}_{k\in\mathbb{N}}$ of $L^2(\Q)$, 
let $P_{\theta,F}^n$ be the law of the sequence $X^n=(X_1^n, X_2^n,\cdots)$ defined by
\begin{align*}
	X_k^n:=\ip{K_\theta F}{e_k}_2 + n^{-1/2}Z_k,
\end{align*}
where the $Z_k$'s are i.i.d. standard normal variables. This gives observations that are statistically equivalent to
the signal-in white noise model \eqref{eq: observation} (for $f=\Phi\circ F$), in the sense that 
the likelihood ratios are equal. For $u\in L^2(\Q)$, we write $\ip{X^n}{u}_2:=\sum_k X^n_k\ip{u}{e_k}_2$, which
series can be shown to converge both in quadratic mean and almost surely.

Let $P_0^n$ be the distribution of the noise sequence $n^{-1/2}(Z_1, Z_2,\cdots)$ (or alternatively of $\dot{W}$ in \eqref{eq: observation}). The log-likelihood has the following expression (e.g. Lemma~L.4 in \cite{GhosalVdV}),
\begin{align}\label{eq: loglikelihood}
	\log\frac{dP^n_{\theta,F}}{dP^n_0}(X^n) = n\ip{X^n}{K_\theta F}_2 - \frac{n}{2}\|K_\theta F\|_2^2.
\end{align}
In combination with a prior on $(\theta,f)$, Bayes's rule yields a posterior distribution $\Pi_n(\cdot\mid X^n)$ for the pair $(\theta,f)$. We are particularly interested in the marginal posterior $B\mapsto \Pi(\theta\in B\mid X^n)$ and its asymptotic behavior as $n\to\infty$. The latter is determined by the semi-parametric structure of the model.

\section{Semi-parametric BvM for the Diffusivity Coefficient}\label{sec: main}

\subsection{Local Asymptotic Normality and Efficient Information}
\label{SectionLAN}

We follow standard semi-parametric theory (for instance as given in \cite{AadStFlour} or Chapter 25 of \cite{AadsBoek}), where we
note that by sufficiency observing $X^n$ is equivalent to observing $n$ i.i.d.\ observations with the law of $X^1$. The starting point is the local asymptotic normality (LAN) of our model. This is shown through expansions of the log-likelihood along one-dimensional submodels of the type $s\mapsto (\theta+s, F+tG)$, for $s\in\mathbb{R}$ and $G\in H$ an arbitrary fixed ``direction''. To obtain this, we first establish the existence of derivatives (or linearisations) for the maps $\theta\mapsto K_\theta F$ and $F\mapsto K_\theta F$. Here we rely on the fact that $\eL_{\theta,F}$ is an isomorphism of $H^{2,1}_{B,0}(\Q)$ onto $L^2(\Q)$ with (continuous) inverse $\eL_{\theta,F}^{-1}: L^2(\Q)\to H^{2,1}_{B,0}(\Q)$. To see this, observe that $\eL_{\theta,F} = L_\theta +\Phi(F)$, with $L_\theta:=\del_t - \frac{\theta}{2}\Delta_x$. We know that $L_\theta:H_{B,0}^{2,1}(\Q)\to L^2(\Q)$ is an isomorphism (c.f. \cite{lionsVol2}, Chapter 4, Remark 15.1). Define then 
$\Lambda:= \eL_{\theta,F}(L_\theta^{-1})$. It is clear that for $I$ the identity operator on $L^2(\Q)$, we have
\[\Lambda u = I u + \Phi(F)L_\theta^{-1}u. \]
Since $\beta>2$, we certainly have $\Phi(F)\in H^{2}(\X)$ and consequently $\Phi(F)L_\theta u \in H^{2,1}(\Q)$, and so $\Phi(F)L_\theta^{-1}:L^2(\Q)\to H^{2,1}(\Q)$ is compact. Furthermore, $\eL_{\theta, F}u = 0$ for $u\in H_{B,0}^{2,1}(\Q)$ implies that $u=0$ by Theorem \ref{thm: existence}. It follows that $\Lambda$ is injective. We conclude by the Fredholm alternative (Theorem VI.6 in \cite{brezis1983}) that it is an isomoprhism. It follows that $\eL_{\theta, F}^{-1} = L_\theta^{-1}\Lambda^{-1}$.

\begin{lemma}\label{lemma: devs}
For all $F\in H$, the map $\theta\mapsto K_\theta F$ from $\Theta$ to $H^{2,1}(\Q)$ 
is Fr\'echet differentiable at every $\theta\in\Theta$ with derivative
\[\dot{K}_\theta F=\eL_{\theta,F}^{-1}\bigl(-\tfrac12\Delta_x K_\theta F\bigr).\] 
The map $F\mapsto K_\theta F$ from $H\subset C(\X)$ to $L^2(\Q)$ is Fr\'echet differentiable  with derivative 
$I_{\theta,F}:H\mapsto L^2(\Q)$ given by 
	\[I_{\theta, F}h := \eL_{\theta,F}^{-1} \bigl(h\Phi'(F)K_{\theta}F\bigr).\]
Furthermore, for every $h\in H, \ \theta\mapsto I_{\theta,F}h$ is continuous.
\end{lemma}

Concretely, the second assertion of the lemma means that 
$\left\| K_{\theta}(F+h) - K_\theta F - I_{\theta,F}h\right\|_2 = o(\|h\|_\infty)$, as $\|h\|_\infty\to 0$.
The LAN property of the one-dimensional submodels is an easy consequence of the lemma.

Let  $I_{\theta,F}^*$ be the adjoint of $I_{\theta,F}: L^2(\X)\to L^2(\Q)$.

\begin{lemma}[LAN expansion]\label{lemma: LAN}
Let $s\in\mathbb{R}$. For a fixed pair $(\theta,F)\in\Theta\times H$, let $\dot{K}_\theta F$ and $I_{\theta,F}$ be as in Lemma~\ref{lemma: devs}. For every $G\in H$, as $n\to\infty$, we have the following convergence in $P^n_{\theta,F}$-probability,
	\begin{align}\label{eq: LANexp}
		\log\frac{dP^n_{\theta+s/\sqrt{n}, F+sG/\sqrt{n}}}{dP^n_{\theta,F}}\to s\ip{\dot{K}_\theta F + I_{\theta,F}G}{\dot{W}}_2 - \frac{1}{2}s^2\|\dot{K}_\theta F + I_{\theta,F}G\|_2^2.
	\end{align}
Assume that there exists $\gamma_{\theta, F}\in H$ such that  $I_{\theta,F}^*\dot K_\theta F = (I_{\theta,F}^*I_{\theta,F}) \gamma_{\theta, F}$.
Then the Fisher information  $\|\dot{K}_\theta F + I_{\theta,F}G\|_2^2$  is minimised at $G = -\gamma_{\theta,F}$. 
If  $\partial^2_t \log K_{\theta}F\neq 0$, then the minimal value
$\Tilde{i}_{\theta,F} := \|\dot{K}_\theta F - I_{\theta,F}\gamma_{\theta,F}\|^2$ is strictly positive.
      \end{lemma}

The number  $\|\dot{K}_\theta F + I_{\theta,F}G\|_2^2$ is the Fisher information for $s$ in the model $(P^1_{\theta+s,F+sG)}: s\in\mathbb{R})$
at $s=0$. The submodel with $G=0$ corresponds to the case where $F$ is known and yields
the ``ordinary'' Fisher information $\|\dot{K}_\theta F\|_2^2$ at $\theta$. The minimal Fisher information across all submodels $\Tilde{i}_{\theta,F}$ is referred to as the "efficient" Fisher information. Under the condition
of Lemma~\ref{lemma: LAN}  that  $I_{\theta,F}^*\dot K_\theta F$ belong to the range of $I_{\theta,F}^*I_{\theta,F}$,  this
is achieved for the submodel with $G = -\gamma_{\theta,F}$, where we can formally write
\begin{align}\label{eq: lfdDef}
  \gamma_{\theta,F}= (I_{\theta,F}^*I_{\theta,F})^{-1}I_{\theta,F}^*\dot K_\theta F.
\end{align} 
Because estimating $\theta$ is hardest in this submodel, the function $\gamma_{\theta,F}$ is
called the \textit{least favourable direction}. The function $I_{\theta,F}\gamma_{\theta,F}$ is the
orthogonal projection of $\dot{K}_\theta F$ onto the closure of $I_{\theta,F}H$ in $L^2(\Q)$,
and we can use Pythagoras rule to rewrite the efficient Fisher information as 
 \[\Tilde{i}_{\theta,F} = \|\dot{K}_\theta F\|_2^2 - \|I_{\theta,F}\gamma_{\theta,F}\|_2^2. \]
 Unless $\gamma_{\theta,F}=0$, the efficient information is strictly smaller than $\|\dot{K}_\theta F\|_2^2$,
 meaning that we lose information for not knowing $F$. If $\dot K_\theta F=0$, then both informations
 are zero, while the efficient information $\Tilde i_{\theta,F}$ also vanishes if  $I_{\theta,F}\gamma_{\theta,F} = \dot K_\theta F$.
 The last assumption of the lemma ensure that these two scenarios do not happen. We note that this is also the identifiability condition in Lemma \ref{lemma: identifiability}. If this assumption fails, then  $\partial^2_t \log K_{\theta}F= 0$ and hence $K_\theta F(x,t)= A(x) e^{B(x)t}$, for
   certain functions $A$ and $B$.
 If the boundary condition $u_0$ is such that the solution to \eqref{eq:PDETorus} takes this exponential form, then the efficient information is zero and the parameter $\theta$ is not estimable at $\sqrt n$-rate
 (\cite{vdV1991}).

In practice, we would like to have concrete assumptions on $\theta, F$ and $u_0$ that guarantee both that $I_{\theta,F}^*\dot K_\theta F$ belongs to the range of the ``information operator'' $I^*_{\theta,F}I_{\theta,F}$, and that $\theta$ is estimable at $\sqrt{n}$-rate. The next theorem deals with those issues. Before stating it, we introduce the Schrödinger operator \[\SSS_{\theta,F}:H^1(\X)\to H^{-1}(\X),\;\;\; h\mapsto \frac{\theta}{2}\Delta h - \Phi(F)h.\]
Since $\Phi(F)\in \mathcal{F}_{\beta,f_{\min}}$, we have that $\Phi(F)\geq f_{\min}>0$. Besides, $\beta>2+d/2>d/2$ implies by Sobolev embeddings that $\|\Phi(F)\|_{\infty}\leq f_{\max}<\infty$. It follows that $\SSS_{\theta,F}:H^{1}(\X)\to H^{-1}(\X)$ is a self-adjoint isomorphism (see the arguments in page 42 of \cite{nickl2025}). In particular, $\SSS_{\theta,F}:H^s(\X)\to H^{s-2}(\X)$ is injective for all $s\geq 1$, and admits an eigenvalue decomposition $\{(\lambda_{\theta,F,j}, e_{j,\theta,F})\}_{j\in\mathbb{N}_0}\subseteq \RR\times H^1(\X)$,
\[\SSS_{\theta,F} h = \sum_{k=0}^\infty \lambda_{\theta,F,j} \ip{h}{e_{\theta,F,j}}_{2}e_{\theta,F, j}.\]
\begin{thm}[Information Theorem]\label{thm:information}
	Let $(\theta,F)\in\Theta\times H$ and let $u_0\in H^{1+\beta}(\X)$. Consider also $\alpha>2+d/2$ and $\xi>\alpha+4$. The following is true.
	\begin{enumerate}[label=(\roman*),leftmargin=*, itemsep=0.5ex, before={\everymath{\displaystyle}}]
		\item If $u_0$ does not lie in an eigenspace of $\SSS_{\theta, F}$, then $\del_t^2\log K_\theta F\neq 0$.
		\item If $F\in H^\xi(\X)$ and $u_0\in H^{1+\xi}(\X)$, then $I^*_{\theta,F}\dot K_{\theta}F\in H^{\xi}(\X)$. Furthermore, the following map is an isomorphism:
		\[\II:=S_{\theta,F}^2I_{\theta,F}^*I_{\theta,F}:H^\alpha(\X)\to H^\alpha(\X).\]
	\end{enumerate}
\end{thm}

Some comments about the results of Theorem \ref{thm:information}. Concerning item (i), certifying that the initial condition $u_0$ does not lie in an eigenspace of $\SSS_{\theta,F}$ is not possible without knowledge of $(\theta, F)$. However since those eigenspaces are finite dimensional, the set of such initial conditions is meagre in $H^{1+\beta}(\X)$; hence the condition holds for generic choices of $u_0$.

The conclusion of item (ii) allows us to derive an expression for the least favourable direction $\gamma_{\theta,F}$ in the following way. Let $\overline\gamma_{\theta,F}:=\SSS_{\theta,F}^2(I^*_{\theta,F}\dot K_\theta F)$. Clearly, for $F\in H^\xi(\X)$, we have that $\overline\gamma_{\theta,F}\in H^{\xi-4}(\X)\subset H^\alpha(\X)$. We claim that $\gamma_{\theta,F} = \II^{-1}\left[\overline\gamma_{\theta, F}\right]$. To see this, observe that this implies
\[\SSS^2_{\theta,F}(I^*_{\theta,F}I_{\theta,F}\gamma_{\theta,F}) = \SSS^2_{\theta,F} (I^*_{\theta,F}\dot K_{\theta}F)\implies I^*_{\theta,F}I_{\theta,F}\gamma_{\theta,F} = I^*_{\theta,F}\dot K_{\theta}F, \]
by injectivity of $\SSS_{\theta,F}$. This matches the definition of $\gamma_{\theta,F}$ in \eqref{eq: lfdDef}, and we can guarantee that $\gamma_{\theta,F}\in H^\alpha(\X)$ for $\alpha$ smaller, but arbitrarily close to $\xi-4$. 

\subsection{Bernstein-von Mises Theorem}

Now assume that the data $X^n$ were generated from $P_{\theta_0,F_0}^n$ for
a given pair of \textit{true} parameters $(\theta_0,F_0)$. We say that the semi-parametric \textit{Bernstein-von Mises theorem holds at $(\theta_0, F_0)$} in $P^n_{\theta_0, F_0}$-probability, if, as $n\to\infty$,
\begin{align}\label{eq: bvm}
	 \Big\|\Pi(\theta \in \cdot \mid X^n) - 
	N\Big(\theta_0 + \frac{1}{\sqrt{n}} \Delta_{ \theta_0, F_0}^n, 
	\frac{1}{n}\tilde i^{-1}_{\theta_0, F_0}\Big)\Big\|_{\text{TV}}\to 0,
	\end{align}
where  $\Delta_{ \theta_0, F_0 }^n$ are measurable
transformations of $X^n$ such that $\Delta_{ \theta_0, F_0 }^n\to N(0, \tilde i^{-1}_{\theta_0, F_0})$ in
distribution, and $\|\cdot\|_{TV}$ is the total variation distance between probability measures. 
Our main result, Theorem~\ref{thm: main} below, shows that this is indeed the case under certain assumptions on $(\theta_0,F_0)$ and on the prior. 

Theorem~\ref{thm: main} is essentially an application of Theorem 12.9 in \cite{GhosalVdV}, which in turn is a rewrite of the result obtained in \cite{castillo}. The proof has two core aspects. The first is to control the remainder in the LAN expansion \eqref{eq: LANexp} uniformly over sieves $\Theta_n\times H_n$ of posterior mass tending to 1. We take these equal to  shrinking balls around $(\theta_0,F_0)$, leverage the contraction result of Proposition~\ref{prop: contract2}, and use various estimates from PDE theory to locally control the remainder. The second aspect is to verify that the prior on $F$ is ``insensitive'' to scaled shifts in the least favourable direction. Using a Gaussian prior, we obtain this easily under the assumption that  the least favourable direction is in the reproducing kernel Hilbert space of the prior on $F$.

We now lay down the assumptions on $(\theta_0,F_0)$ and on the prior. We assume that $\theta_0$ is an interior point of $\Theta$. We further consider parameters $\xi,\alpha$ such that $\beta+d/2<\alpha < \xi-4$, and assume that $F_0\in H^\xi(\X)$, and that $u_0\in H^{1+\xi}(\X)$ and is not in an eigenspace of $\SSS_{\theta,F}$. 

We choose any prior $\pi_\theta$ on $\Theta$ with a continuous strictly positive density relative to Lebesgue measure, and
endow $F$ with a re-scaled Gaussian process prior supported on $H$. The re-scaling biases the base prior towards functions with small $H^\beta(\X)$-norm, which helps to obtain contraction of the posterior measure, as was noted in  \cite{monard2021} and
has become standard in the literature on nonlinear Bayesian inverse problems (see also \cite{Nickl23}). 
We first select a base prior $\pi'_F$ that is supported on $H$ and has reproducing kernel Hilbert space (RKHS) $\HHH$ 
contained in $H^\alpha(\X)$. The prior $\pi_F=\pi_{n,F}$ of $F$  is then the distribution of 
\begin{align}\label{eq: priorDef}
	F = n^{-d/(4\alpha + 2d + 8)}F',\qquad \ F'\sim \pi_F'.
\end{align}
We can now state our main result. For technical ease, we will only consider the case $d\le 3$.

\begin{thm}[Semiparametric BvM]\label{thm: main} 
For $d\le 3$, let $\beta+d/2<\alpha < \xi-4$ and consider the prior $\Pi=\pi_\theta\times\pi_{n,F}$ with $\pi_{n,F}$ given by \eqref{eq: priorDef}. 
Suppose that $\theta_0$ is an interior point of $\Theta$ and that $F_0\in \HHH \cap H^\xi(\X)$. Suppose also that $u_0\in H^{1+\xi}(\X)$ does not lie in an eigenspace of $\SSS_{\theta,F}$, and that $\gamma_{\theta_0,F_0}\in \HHH$. If $\beta> 2+d$, then the Bernstein-von Mises theorem \eqref{eq: bvm} holds at $(\theta_0,F_0)$.
\end{thm}

The BvM in Theorem~\ref{thm: main} can be made more concrete by specifying an explicit base prior $\pi'_F$. To this end, let
$(e_k)_{k\ge 0}$ denote the $L^2(\X)$-orthonormal basis of eigenfunctions of the periodic Laplacian $\Delta$. We have 
\[\Delta e_j = -\lambda_j e_k, \qquad j\ge 0,\]
with $e_0=1, \lambda_0=0$  and $0<\lambda_k\leq \lambda_{k+1}\asymp k^{2/d}$ as $k\to\infty$. An equivalent sequence space norm on $H^\alpha(\X)$ is given by
\[\|h\|_{h^\alpha}^2 := \sum_{k\ge 0} (1+\lambda_k)^\alpha|\ip{h}{e_k}_{2}|^2.\]
Restricting the $e_k$'s to their real parts, we define the base prior $\pi'_F$ as the law of the Gaussian random series
\[F'(x)=\sum_{k\ge 0}(1+\lambda_k)^{-\alpha/2} Z_k e_k(x),
\;\;\ Z_k \sim^{iid} N(0,1).\]
It is well known that if $\beta<\alpha-d/2$, the sample paths of $F'$ belong to $H^\beta(\X)$ almost surely, and that the reproducing kernel Hilbert space $\HHH$ of $\pi'_F$ coincides with $H^\alpha(\X)$. The prior $\pi_{n,F}$ is then defined via the rescaling \eqref{eq: priorDef}.

\begin{cor}\label{corrolary:main}
	Let $\xi-4>\alpha>\beta+d/2$, with $\beta > 2+d$. Assume that $\Pi=\pi_\theta\times\pi_{n,F}$ is such that $\pi_\theta$ has a strictly positive density with respect to the Lebesgue measure and that $\pi_{n,F}$ is obtained from the Gaussian series prior $\pi'_F$ described above. Let $\theta_0$ be an interior point of $\Theta$. If $F_0\in H^\xi(\X)$, and if $u_0\in H^{1+\xi}(\X)$ does not lie in an eigenspace of $\SSS_{\theta_0,F_0}$, then the Bernstein-von Mises theorem holds at $(\theta_0,F_0)$.
\end{cor}

\begin{proof}
	Note that the assumption $F_0\in H^\xi(\X)$ implies, in view of the discussion following Theorem~\ref{thm:information}, that $\gamma_{\theta_0,F_0}\in H^\alpha(\X)$. Since the reproducing kernel Hilbert space $\HHH$ of the base prior $\pi'_F$ coincides with $H^\alpha(\X)$ in this case, both $F_0$ and $\gamma_{\theta_0,F_0}$ lie in $\HHH$. The result follows by applying Theorem~\ref{thm: main}.
\end{proof}

%\subsection{Example with Whittle-Mat\'ern Prior}
%\label{sec: example}
%
%The BvM in Theorem \ref{thm: main} can be made more concrete by specifying an explicit base prior $\pi'_f$. To this end, consider a Whittle-Mat\'ern process $M = \{M(x): x\in \X\}$ of regularity $\alpha-d/2$ (see Example~11.8 in \cite{GhosalVdV}). The RKHS of $M$ is known to be $H^\alpha(\X)$ and for $\beta<\alpha-d/2$, the sample paths of $M$ belong to $H^\beta(\X)$ almost surely. The following specialisation Theorem~\ref{thm: main} holds.
%
%\begin{thm}
%	Let $\xi-4>\alpha>\beta+d/2$, with $\beta>2+d$. Assume that $\Pi=\pi_\theta\times\pi_f$ is such that $\pi_\theta$ has a positive density with respect to the Lebesgue measure and that $\pi_f$ is a re-scaled version of the Whittle-Matérn prior $\pi'_f=M$ described above.  Let $\theta_0$ be an interior point of $\Theta$. If $F_0\in H^\xi(\X)$ and $u_0\in H^{1+\xi}(\X)$ does not lie in an eigenspace of $\SSS_{\theta_0, F_0}$, then the BvM holds at $(\theta_0,F_0)$.Assume also that  
%\end{thm} 
%
%\begin{proof}
%	
%	Note that the assumption that $F\in H^\xi(\X)$ directly implies (in view of the discussion after Theorem \ref{thm:information}) that $\gamma_{\theta,F}\in H^\alpha(\X)$. Since the RKHS $\HHH$ of $\pi'_f$ is equal to $H^\alpha(\X)$ in this case, the conclusions of Theorem \ref{thm: main} hold.
%\end{proof}

\subsection{Contraction Results}\label{sec:contraction}

A key component of the proof of Theorem~\ref{thm: main} is the contraction of the posterior measure at the true parameter $(\theta_0,F_0)$. 
This is obtained in two steps: contraction in the direct problem, followed by contraction in the inverse problem. 

\begin{prop}[Contraction in the Direct Problem]\label{prop: contract1}
Consider the prior $\Pi_n=\pi_\theta\times \pi_{n,F}$ with $\pi_{n,F}$ given by \eqref{eq: priorDef}. Assume that $F_0\in \HHH \cap H$ and that $\theta_0$ is an interior point of $\Theta$. 
Let $\delta_n = n^{-(2+\alpha)/(2\alpha+4+d)}$ and for sufficiently large $M>0$ consider
	\begin{align}
		H_n = \{F\in H: F = F_1+F_2, \|F_1\|_{(H^2(\X))^*} \leq M\delta_n, \|F_2\|_{\HHH}\leq M, \|F\|_{H^\beta(\X)}\leq M \},
	\end{align}
Then for every sufficiently large constant $m$, the following convergence is true in $P^n_{\theta_0,F_0}$-probability:
	\begin{align}
		\Pi_n\bigl((\theta,F)\in \Theta\times H_n : \|K_\theta F- K_{\theta_0} F_0\|_2 < m\,\delta_n \mid X^n \bigr) &\gaatp 1.
	\end{align}
\end{prop}

\begin{prop}[Contraction]\label{prop: contract2}
Let $\Pi_n, F_0, \theta_0, \delta_n$ and the sets $H_n$ be as in Proposition~\ref{prop: contract1}. Assume that $u_0\in H^{1+\beta}(\X)$ does not lie in an eigenspace of $\SSS_{\theta,F}$.
Then for every sufficiently large constant $\ell$, the following convergence holds in $P^n_{\theta_0,F_0}$-probability:
	\begin{align}
		\Pi_n\bigl((\theta,F)\in \Theta\times H_n : |\theta-\theta_0|+\| F- F_0\|_2 < \ell\, \delta_n^{\beta/(2+\beta)} \mid X^n \bigr) &\gaatp 1.
	\end{align}
\end{prop}

\begin{proof}
The result follows from Proposition~\ref{prop: contract1}  in conjunction with Lemma~\ref{lemma: stability} below,
where we take $\ell = Lm^{\beta/(\beta+2)}$.
\end{proof}

\begin{lemma}[Stability Estimate]\label{lemma: stability}
	Let $B_{H^\beta}(R):=\{F\in H^\beta(\X): \|F\|_{H^\beta}\leq R\}$ be the ball  in $H^\beta(\X)$ of radius $R>0$. Assume that $u_0\in H^{1+\beta}(\X)$ does not lie in an eigenspace of $\SSS_{\theta,F}$.
	For all $R>0$, there exists a constant $L>0$ such that for all  small enough $\delta>0$,
	\begin{align*}
		\sup\{|\theta-\theta_0| + \|F-F_0\|_2 : (\theta,F)\in \Theta\times B_{H^\beta}(R), \ \|K_\theta F - K_{\theta_0}F_0\|_2
		\leq \delta\}\leq L\,\delta^{\beta/(\beta+2)}.
	\end{align*}
\end{lemma}

\section{Proof of the Semi-parametric BvM (Theorem \ref{thm: main})}

\label{sec:proofBvM}

The theorem follows from Theorem~12.9 in \cite{GhosalVdV}, which is an adaptation of results by \cite{castillo}.
We apply Theorem~12.9 with the least favorable transformation $(\theta,F)\mapsto (\theta_0,F+(\theta-\theta_0)\gamma_{\theta_0,F_0})$. 
To ease notation we write $\gamma$ for $\gamma_{\theta_0,F_0}$.

Since we put a Gaussian prior on $F$ and $\gamma\in\HHH$, the prior shift condition (12.14) of Theorem~12.9 is shown to be satisfied similarly to Example~12.11 in the same reference. For the sake of completeness, we prove it at the end of this section (see \ref{sec:shift}).

The posterior consistency conditions in Theorem~12.9 of \cite{GhosalVdV}, 
can all be obtained from the statement of Proposition~\ref{prop: contract2}.
Thus we need only verify (12.13) in \cite{GhosalVdV}. 
Here in view of Proposition~\ref{prop: contract2} we may assume that the sets $\Theta_n\times H_n$ are contained in shrinking balls 
$\{(\theta,h)\in \Theta\times H: |\theta-\theta_0|<\eps_n,\|h-F_0\|_2<\eps_n\}$ of $(\theta_0,F_0)$, 
for the rate $\eps_n\asymp \delta_n^{\beta/(\beta+2)}$, where $\delta_n$ is as in Proposition~\ref{prop: contract1}. 
By the same proposition we may also assume that $H_n\subset \{h\in H: \|h\|_{H^\beta} < M\}$. 
Note also that $\gamma\in H$ by assumption that $\HHH\subseteq H^\alpha(\X)\subset H$. 

Straightforward computations using \eqref{eq: loglikelihood} and the fact that for any $F\in H, \ \dot K_{\theta_0}F-I_{\theta_0,F}\gamma_{\theta_0,F}$ is
orthogonal to the range of $I_{\theta_0,F}$
(whence the identity $\ip{I_{\theta_0,F}(F-F_0)}{\dot K_{\theta_0}F-I_{\theta_0,F}\gamma}_2
=\ip{I_{\theta_0,F}(F-F_0)}{I_{\theta_0,F}(\gamma_{\theta_0,F}-\gamma)}_2$
yield
\[
\log \frac{{dP_{\theta,F}^n}}{dP_{\theta_0, F+(\theta-\theta_0)\gamma}^n}(X^n)
= \sqrt{n}(\theta-\theta_0) G_{\theta_0}(F,\gamma) - \frac{n}{2}|\theta-\theta_0|^2 \Tilde i_{\theta_0, F}(\gamma) + R_n(\theta,F),
\]
where, for $\dot W=\sqrt n(X^n-K_{\theta_0}F_0)$,
\begin{align*}
	G_{\theta_0}(F,h)&= \ip{\dot W}{\dot K_{\theta_0}F - I_{\theta_0,F} h}_2,\\
	\Tilde i_{\theta_0, F}(h)&= \|\dot K_{\theta_0}F\|_2^2 - \|I_{\theta_0,F}h\|_2^2 ,\\
	R_n(\theta,F) = \sqrt{n}&\ip{\dot W}{K_{\theta}F-K_{\theta_0}F - (\theta-\theta_0)\dot K_{\theta_0}F}_2\\
	&-\sqrt{n}\ip{\dot{W}}{K_{\theta_0}(F+(\theta-\theta_0)\gamma) - K_{\theta_0}F - (\theta-\theta_0)I_{\theta_0,F}\gamma}_2\\
	&- n\ip{K_{\theta}F-K_{\theta_0}F - (\theta-\theta_0)\dot K_{\theta_0}F}{ K_{\theta_0}F-K_{\theta_0}F_0}_2\\
	& +n\ip{K_{\theta_0}(F+(\theta-\theta_0)\gamma)-K_{\theta_0}F - (\theta-\theta_0) I_{\theta_0,F}\gamma}{K_{\theta_0}F-K_{\theta_0}F_0}_2\\
	& +\frac{n}{2}(\theta-\theta_0)^2\|\dot K_{\theta_0}F\|_2^2- \frac{n}{2}\|(K_{\theta}-K_{\theta_0})F\|_2^2\\
	& -\frac{n}{2}(\theta-\theta_0)^2\|I_{\theta_0,F}\gamma\|_2^2+ \frac{n}{2}\|K_{\theta_0}(F+(\theta-\theta_0)\gamma)-K_{\theta_0}F\|_2^2\\
	&- n(\theta-\theta_0)\ip{K_{\theta_0}F - K_{\theta_0}F_0 - I_{\theta_0,F}(F-F_0)}{\dot K_{\theta_0}F - I_{\theta_0,F}\gamma}_2\\
	&+n(\theta-\theta_0)\ip{I_{\theta_0,F}(F-F_0)}{I_{\theta_0,F}(\gamma-\gamma_{\theta_0,F})}_2.
\end{align*}
Therefore, condition (12.13) of Theorem~12.9 in \cite{GhosalVdV} is satisfied, with $\tilde G_n:=G_{\theta_0}(F_0,\gamma)$
and $\tilde i_n=\tilde i_{\theta_0,F_0}$, if 
\begin{align}
	\sup_{(\theta,F)\in \Theta_n\times H_n} \frac{\sqrt{n}(\theta-\theta_0)| G_{\theta_0}(F, \gamma)-G_{\theta_0}(F_0,\gamma)|}{1+n(\theta-\theta_0)^2} &\gaatp 0,\label{eq: LANapprox}\\
	\sup_{(\theta,F)\in \Theta_n\times H_n} \frac{n(\theta-\theta_0)^2|\Tilde i_{\theta_0,F} (\gamma)-\tilde i_{\theta_0, F_0}|}{1+n(\theta-\theta_0)^2} &\to 0,\label{EqContFI}\\
	\sup_{(\theta,F)\in \Theta_n\times H_n} \frac{|R_n(\theta,F)|}{1+n(\theta-\theta_0)^2} &\gaatp 0.\label{eq: remaindercondition}
\end{align}
The Bernstein-von Mises theorem is then satisfied at $(\theta_0,F_0)$ with 
$\Delta_{\theta_0, F_0}^n:= \Tilde i_{\theta_0,F_0}^{-1}G_{\theta_0}(F_0, \gamma)$, which is exactly $N(0, \Tilde i_{\theta_0,F_0}^{-1})$
distributed. (Note that the statement of Theorem~12.9 in \cite{GhosalVdV} needs an extra factor $n^{-1/2}$, 
as acknowledged in the list of errata \cite{vaartghosal_errata}.) To prove \eqref{eq: LANapprox}--\eqref{eq: remaindercondition}, we rely on the results of Lemmas~\ref{lemma: bounds} and \ref{lemma: refinedBounds}.

Since $\sqrt{n}|\theta-\theta_0|\leq 1+n(\theta-\theta_0)^2$, for \eqref{eq: LANapprox} it suffices
to show that $\sup_{F\in H_n}|G_{\theta_0}(F,\gamma) - G_{\theta_0}(F_0,\gamma)|$ tends to zero
in probability. This is the supremum of the sum of the Gaussian processes
	$\ip{\dot W}{\dot K_{\theta_0}F - \dot K_{\theta_0}F_0}_2$  and $\ip{\dot W}{I_{\theta_0,F}\gamma -
		I_{\theta_0,F_0}\gamma }_2$.
	The first is indexed by the sets
$\mathcal{T}_n=\{\dot K_{\theta_0}F - \dot K_{\theta_0}F_0: F\in H_n\}$ and has
intrinsic metric $d(F,G)= \|\dot K_{\theta_0}F-\dot K_{\theta_0}G\|_2^2$, for $F,G\in H_n$.
By Lemma~\ref{lemma: bounds} (iv) the diameter
$\sup_{F,G\in H_n} d(F,G)$ of the set $\mathcal{T}_n$ relative to this metric
is bounded above by  a multiple of $\eps_n$, since $H_n\subset \{h: \|h-F_0\|_2<\eps_n\}$ by construction.
Moreover, by Lemma~\ref{lemma: refinedBounds} (i) the set $\mathcal{T}_n$ belongs to a multiple of
the unit ball in the space $H^{2+\eta,1+\eta/2}(\Q)$, since also $H_n\subset \{h : \|h\|_{H^\beta(\X)}\le M\}$
by construction. The logarithm of the $\eps$-covering numbers of the unit ball of $H^{2+\eta,1+\eta/2}(\Q)$ relative to the
$L_2(\Q)$-metric are bounded above by a multiple of $(1/\eps)^{d/(2+\eta)+1/(1+\eta/2)}$. We may apply  Dudley's bound (\cite{Dudley73}, or  Corollary~2.2.9 in \cite{vdV_wellner2023}) to get
\begin{align}
	\EE\sup_{F\in H_n}|\ip{\dot W}{\dot K_{\theta_0}F - \dot K_{\theta_0}F_0}| &\lle \int_0^{\eps_n}\sqrt{\log{N(\eps,\{u: \|u\|_{H^{2+\eta,1+\eta/2}(\Q)}\leq 1\}, \|\cdot\|_2)}}\,d\eps\nonumber\\
	&\lle \int_0^{\eps_n} \left(\dfrac{1}{\eps}\right)^{(\frac{d}{2+\eta} + \frac{1}{1+\eta/2})\frac{1}{2}}\,d\eps. \label{eq: dudley}
\end{align}
Whenever $d < 2+2\eta$, the last integral is bounded by a constant times $\eps_n$ and hence tends to 0. Choosing $\eta\in (1/2,1)$, we have that this is certainly verified for $d\le 3$. The second Gaussian process is indexed by $\mathcal{T}_n:=\{I_{\theta_0,F}\gamma - I_{\theta_0,F_0}\gamma: F\in H_n\}$. We bound its supremum with the same argument; this time using item (ii) of Lemma \ref{lemma: bounds} instead of item (iv) to bound the diameter by a multiple of $\eps_n$, and using item (iv) of Lemma \ref{lemma: refinedBounds} to show that $\mathcal{T}_n$ belongs to a multiple of the unit ball in $H^{2+\eta,1+\eta/2}$. This concludes the proof of \eqref{eq: LANapprox}.

To bound \eqref{EqContFI} it suffices to show that $|\tilde i_{\theta_0,F}(\gamma) - \tilde i_{\theta_0,F_0}|$
tends to 0 uniformly  over $H_n$. Now for $F\in H_n$,
\begin{align*}
	&|\tilde i_{\theta_0,F}(\gamma) - \tilde i_{\theta_0,F_0}|
	= \left| \|\dot K_{\theta_0}F\|^2_2 - \|\dot K_{\theta_0} F_0\|^2_2 + \|I_{\theta_0,F_0}\gamma\|^2_2 - \|I_{\theta_0,F}\gamma\|^2_2 \right| \\
	&=\Bigl|\ip{\dot K_{\theta_0}F - \dot K_{\theta_0}F_0}{\dot K_{\theta_0}F + \dot K_{\theta_0} F_0} + \ip{I_{\theta_0,F_0}\gamma - I_{\theta_0,F}\gamma}{I_{\theta_0,F_0}\gamma + I_{\theta_0,F}\gamma}\Bigr|\\
	&\le \|\dot K_{\theta_0}F - \dot K_{\theta_0}F_0\|\left(\|\dot K_{\theta_0}F\|+\|\dot K_{\theta_0}F_0\|\right) + \|I_{\theta_0,F_0}\gamma - I_{\theta_0,F}\gamma\|\Bigl(\|I_{\theta_0,F_0}\gamma\| + \|I_{\theta_0,F}\gamma\|\Bigr)\\
	&\lle \|F-F_0\|_2, 
\end{align*}
in view of Lemma~\ref{lemma: bounds} (iv) and (ii), where the constant in the last inequality depends on $M$. For $F\in H_n$,
the right side is bounded by $\eps_n$ and hence tends to zero.

Finally we show that \eqref{eq: remaindercondition} holds by dealing with each of the eight terms in 
$R_n(\theta,F)$ separately. 

Using that $\sqrt{n}/(1+n(\theta-\theta_0)^2)\leq 1/|\theta-\theta_0|$, we can bound the ratio of the first term with $1+n(\theta-\theta_0)^2$ by $\ip{\dot W}{G_{\theta,F}}_2$ where,
\[G_{\theta,F}:= \frac{K_\theta F-K_{\theta_0}F-(\theta-\theta_0)\dot K_{\theta_0}F}{\theta-\theta_0}.\]
This is a Gaussian process indexed by $\mathcal{G}_n = \{G_{\theta,F} : (\theta,F)\in \Theta_n\times H_n\}$. By Lemma~\ref{lemma: refinedBounds} (ii), the $H^{2+\eta,1+\eta/2}(\Q)$-norm of $G_{\theta,F}\in\mathcal{G}_n$ is bounded above by a multiple of $|\theta-\theta_0|$, which is bounded above by a multiple constant of $\eps_n$ uniformly in $H_n$. The same is true for the $L^2$-norm. Furthermore, the intrinsic metric of $\ip{\dot W}{G}_2$ is the $L^2$-norm of $G$. We can thus use Dudley's bound again (c.f. \eqref{eq: dudley}) to show that $\sup_{G\in\mathcal{G}_n} \bigl| \ip{\dot W}{G}_2 \bigr|$ tends to 0 in probability for $d\in\{1,2,3\}$ by the same argument as in the proof of \eqref{eq: LANapprox}.

The second term is dealt with similarly by using Lemma~\ref{lemma: refinedBounds} (iii) instead of (ii), and by bounding the $L^2$-diameter by a constant multiple of $\eps_n$ using item (vi) of Lemma \ref{lemma: bounds}. 

By the Cauchy-Schwartz inequality and (i) and (v) of Lemma~\ref{lemma: bounds}, the third term is bounded above by a multiple of $n(\theta-\theta_0)^2\|F\|_2\|F-F_0\|_2(1+M^2)$, which is $o(1+n(\theta-\theta_0)^2)$ uniformly in $H_n$.

By a similar argument using item (vi) instead of (v), the fourth term can also be shown to be of order $o(1+n(\theta-\theta_0)^2)$ uniformly over $H_n$.

By items (iii) and (v) of Lemma`\ref{lemma: bounds}, the fifth term is bounded above by a multiple of $n|\theta-\theta_0|^3$ uniformly over $\Theta_n\times H_n$, which is $o(1+n(\theta-\theta_0)^2)$. 

By items (i) and (vi) of Lemma`\ref{lemma: bounds}, the sixth term is bounded above by a multiple of $n|\theta-\theta_0|^3$ uniformly over $\Theta_n\times H_n$, which is $o(1+n(\theta-\theta_0)^2)$.

We bound the ratio of the seventh term with $1+n(\theta-\theta_0)^2$ over $H_n$ by,
\begin{align*}
	&\sqrt{n}\bigl|\ip{K_{\theta_0}F - K_{\theta_0}F_0 - I_{\theta_0,F}(F-F_0)}{\dot K_{\theta_0}F - I_{\theta_0,F}\gamma}_2\bigr|\\ 
	&\qquad\leq \sqrt{n}\|K_{\theta_0}F - K_{\theta_0}F_0 - I_{\theta_0,F}(F-F_0)\|_2\, (\|\dot K_{\theta_0}F\|_2 + \|I_{\theta_0,F}\gamma\|_2)\\
	&\qquad\lle \sqrt{n}\|F-F_0\|_2\|F-F_0\|_\infty.
\end{align*}
by Lemma~\ref{lemma: bounds} (vi). We show that this tends to zero uniformly in $F\in H_n$ at the end of
the proof.

To bound the last term's ratio with $1+n(\theta-\theta_0)^2$, we use the definition of
$\gamma_{\theta_0,F}$ to split it into the following four terms:
\begin{align*}
	&\sqrt{n}\ip{I_{\theta_0,F_0}(F-F_0)}{I_{\theta_0,F_0}\gamma -\dot K_{\theta_0}F_0}_2\\
	&\qquad+\sqrt{n}\ip{I_{\theta_0,F_0} (F-F_0)}{(I_{\theta_0,F}-I_{\theta_0,F_0})\gamma}_2\\
	&\qquad-\sqrt{n}\ip{I_{\theta_0,F_0} (F-F_0)}{\dot K_{\theta_0}F - \dot K_{\theta_0}F_0}_2\\
	&\qquad+\sqrt{n}\ip{(I_{\theta_0,F}-I_{\theta_0,F_0})(F-F_0)}{I_{\theta_0,F}\gamma - \dot K_{\theta_0}F}_2.
\end{align*}
By definition of $\gamma=\gamma_{\theta_0,F_0}$, the first term is equal to 0.
The second and third terms are bounded by a multiple of $\sqrt{n}\|F-F_0\|^2_2$ by
(vii) and (iv) of Lemma~\ref{lemma: bounds}, respectively.
The fourth term is bounded by a multiple of $\sqrt{n}\|F-F_0\|_2\|F-F_0\|_\infty$
by item (ii) of the same lemma.

To complete the proof of the theorem we show that $\sqrt{n}\sup_{H\in H_n}\|F-F_0\|_2\|F-F_0\|_\infty$
tends to zero in probability. Since $\beta>d/2$, we can choose $\beta'\in (d/2,\beta)$
and apply the interpolation inequality \eqref{EqInterpolation}
\[\|F-F_0\|_{H^{\beta'}}\lle \|F-F_0\|_2^{1-\beta'/\beta}\|F-F_0\|_{H^\beta}^{\beta'/\beta},\]
to obtain that $\|F-F_0\|_{H^{\beta'(\X)}} \lle \eps_n^{1-\beta'/\beta}$ uniformly in $F\in H_n$.
Then by Sobolev embedding $\sup_{F\in H_n}\|F-F_0\|_\infty$ tends to zero at the same order.
Thus $\sqrt{n}\sup_{H\in H_n}\|F-F_0\|_2\|F-F_0\|_\infty$ is of the order $\sqrt{n}\eps_n^{2-\beta'/\beta}$.
For $\eps_n=\delta_n^{\beta/(2+\beta)}$ and $\delta_n$ given in Proposition~\ref{prop: contract1},
this tends to zero if
	\[\alpha(2-(\beta-\beta')) < 2(\beta-\beta') - \frac{d}{2}\beta - d - 4 .\]
	Since we have in the statement of the theorem that $\beta>2+d$, it is always possible to select $\beta'\in (d/2,\beta)$ for which $\beta-\beta'=2+d/2$. The inequality above then reduces to $\alpha>\beta$, which is true by assumption. This concludes the proof of Theorem~\ref{thm: main}. %The true theoretical upper bound is not 2+d. Since beta is integer valued it actually is beneficial to include the case  $beta\geq 2+d$ but whatever. 

\begin{remark}
	The condition $d\leq 3$ came in handy when controlling the entropy integral of the unit ball in $H^{2+\eta,1+\eta/2}(\Q)$ because we could only apply Lemma \ref{lemma: refinedBounds} for $\eta<1$. Theoretically (in view of \eqref{eq:refinedCIterated} and Sobolev embedding), the true upper bound on $\eta$ for which Lemma \ref{lemma: refinedBounds} is expected to hold is $\beta-d/2$. We did not pursue this here to simplify the proof of Lemma \ref{lemma: refinedBounds}.
\end{remark}

\subsection{Proof of the Prior Shift Condition}
\label{sec:shift}

We again denote $\gamma=\gamma_{\theta_0,F_0}$. With our choices of least favourable transformations $(\theta,F)\mapsto (\theta_0, F+(\theta-\theta_0)\gamma)$, condition (12.14) in \cite{GhosalVdV} reduces to the following condition on the prior on $F$:
\begin{align*}
	\sup_{\theta\in\Theta_n, F\in H_n} \frac{|\log (d\pi_{n,F+(\theta-\theta_0)\gamma}/d\pi_{n,F}(F))|}{1+n(\theta-\theta_0)^2}\to 0,
\end{align*}
with $\Theta_n\times H_n$ having posterior mass tending to 1. Denote by $\mathcal{H}$ the RKHS of $\pi_{n,F}$. By construction, $\|\cdot\|_{\mathcal{H}} = \sqrt{n}\delta_n\|\cdot\|_\HHH$, and so $\gamma\in\HHH\implies \gamma\in\mathcal{H}$. By Cameron-Martin's formula, we then have that
\[\frac{d\pi_{F+(\theta-\theta_0)\gamma}}{d\pi_F}(F) = e^{(\theta-\theta_0)\|\gamma\|_\mathcal{H} U(F) - (\theta-\theta_0)^2\|\gamma\|_\mathcal{H}^2},\]
with $U(F)$ some measurable transformation of $F$ with a standard normal distribution. Consider now $\Theta_n=\Theta$ and $H_n=\{F\in H: |U(F)|\leq \sqrt{2Cn}\delta_n\}$ for some $C>0$ to be determined. Over $\Theta\times H_n$, we thus have
\begin{align*}
	\frac{|\log (d\pi_{n,F+(\theta-\theta_0)\gamma}/d\pi_{n,F}(F))|}{1+n(\theta-\theta_0)^2}&\lle \frac{n(\theta-\theta_0)\delta_n^2\|\gamma\|_\HHH}{1+n(\theta-\theta_0)^2} + \frac{n(\theta-\theta_0)^2\delta_n^2\|\gamma\|^2_\HHH}{1+n(\theta-\theta_0)^2}\\
	&\leq \sqrt{n}\delta_n^2\|\gamma\|_\HHH + \delta_n^2\|\gamma\|_\HHH^2.
\end{align*}
The right and side goes to 0 provided that $\delta_n << n^{-1/4}$, which is true for $\alpha+2>d/2$. This is certainly the case by assumption on $\alpha$. 

The proof is concluded upon showing that the selected $\Theta\times H_n$ has posterior mass tending to 1. To this end, we apply the remaining mass principle (Theorem 8.20 in \cite{GhosalVdV}). Observe that the Kullback-Leibler divergence $K(P^n_{\theta_0,F_0}, P^n_{\theta,F})$ and the variance $V_{2,0}(P^n_{\theta_0,F_0}, P^n_{\theta,F})$ are equal to 1/2 and 1 times $n\|K_{\theta}F-K_{\theta_0}F_0\|^2_2$ respectively. By item (i) in the proof of Proposition \ref{prop: contract1}, it follows that 
\[\Pi((K\vee V_{2,0})(P^n_{\theta_0,F_0}, P^n_{\theta,F})\leq n\delta_n^2)\geq \Pi(\|K_{\theta}F-K_{\theta_0}F_0\|_{2}\leq \delta_n)\geq e^{-An\delta_n^2}.\]
Furthermore, the Gaussian tail bound for the normal distribution gives us that $\Pi((\Theta\times H_n)^C)\leq e^{-(C^2-1)n\delta_n^2}$. We select $C^2>A+1$ and the remaining mass principle can thus be applied to conclude that the posterior mass of $\Theta\times H_n$ must tend to 1. 

\section{Proofs of Posterior Contraction and Stability Estimate (Section \ref{sec:contraction})}

\label{sec:proofContraction}

\subsection{Proof of Proposition \ref{prop: contract1}}

The theorem is a consequence of the main result of Theorem~6 in \cite{vaartghosal2007} (or see
Theorem~8.31 in \cite{GhosalVdV} combined with Theorem~8.20,
or Theorem~1.3.2 in \cite{Nickl23}). Both the Kulback-Leibler divergence
and the testing metric for the white noise model are dominated by the metric
$d_T((\theta_1,F_1),(\theta_2,F_2)) := \|K_{\theta_1}F_1 - K_{\theta_2}F_2\|_2$. Thus it suffices
to verify the three conditions:
\begin{enumerate}[label=(\roman*),leftmargin=*, itemsep=0.5ex, before={\everymath{\displaystyle}}]
	\item $\Pi_n\bigl((\theta,F)\in\Theta\times H: d_T((\theta,F), (\theta_0,F_0))\leq \delta_n\bigr) \geq e^{-An\delta_n^2}$, for some constant $A>0,$
	\item $\Pi_n\bigl((\Theta\times H_n)^C\bigr)\leq e^{-Bn\delta_n^2}$, for some $B>A+2,$
	\item $\log N(\overline{m}\delta_n,\Theta\times H_n, d_T )\leq n\delta_n^2$ for all sufficiently large $\overline{m}$.
\end{enumerate}
For our Gaussian prior these conditions can be verified following the method of
\cite{vdVvZ2007}, \cite{vdVvZ2008}. We follow the presentation in 
the proof of Theorem 2.2.2 in \cite{Nickl23}, which is suited to Sobolev space.
We will be assuming $\kappa =2$ in the definition of the prior and of $\delta_n$ when mentioning this proof.

Starting with (ii), observe that $\Pi_n\bigl(((\Theta\times H_n)^C)\bigr)=\pi_{n,F}(H_n^C)$ which can subsequently be bounded above following steps i) and iii) in the proof of Theorem 2.2.2 in \cite{Nickl23}.
Note that step iii) requires point (i) to be proven which we do right below. 

For (i), we use the following Lipschitz type condition, whose proof can be found at the end of this section.

\begin{lemma}[Lipschitz Condition]\label{lemma: lipschitz}
	Let $\beta>2+d/2$. Let $B_{H^\beta}(R)$ be the ball  in $H^\beta(\X)$ of radius $R>0$. There exists a constant  $C=C(R)$ dependent on $R$ such that  for all $\theta_1,\theta_2\in\Theta
	$ and $F_1, F_2\in B_{H^\beta}(R)$,
	\begin{align}
		\|K_{\theta_1}F_1 - K_{\theta_2}F_2\|_2  \leq C(R)\left(|\theta_1-\theta_2| + \|F_1-F_2\|_{(H^2(\X))^*} \right),\label{eq: lipschitz}
	\end{align} 
\end{lemma}

% By \eqref{eq: uniformBound}  $\|K_{\theta}F\|_\infty\leq U$ is always true over $\Theta\times H$.
Since $F_0\in H^\beta(\X)$, the triangle inequality gives that $\|F\|_{H^\beta}\leq M':=2M$ if 
$\|F-F_0\|_{H^\beta} \leq M$  and $M\ge \|F_0\|_{H^\beta} $.
Using Lemma~\ref{lemma: lipschitz}, it follows that,
\begin{align*}
	&\Pi_n\bigl((\theta,F)\in \Theta\times H: \|K_\theta F - K_{\theta_0}F_0\|_2 \leq \delta_n\bigr)\\
	%	&\geq \Pi_n\bigl((\theta,F): \|K_\theta F - K_{\theta_0}F_0\|_2\leq \delta_n, \|F-F_0\|_{H^\beta}\leq M\bigr)\\
	&\geq \Pi_n\bigl((\theta,F): |\theta-\theta_0| + \|F-F_0\|_{(H^2(\X))^*}\leq \delta_n/C(M'), \|F-F_0\|_{H^\beta}\leq M\bigr)\\
	%	&\geq \Pi_n\bigl((\theta,F): |\theta-\theta_0|\leq \delta_n/2C(M'),\  \|F-F_0\|_{(H^2(\X))^*}\leq \delta_n/2C(M'), \|F-F_0\|_{H^\beta}\leq M\bigr)\\
	&=\pi_\theta \bigl(|\theta-\theta_0|\leq \delta_n/2C(M')\bigr)\cdot \pi_{n,F}\bigl(\|F-F_0\|_{(H^2(\X))^*}\leq \delta_n/2C(M'), \|F-F_0\|_{H^\beta}\leq M\bigr).
\end{align*}
Since the density of $\pi_\theta$ is assumed to be bounded away from 0, the first factor in the right side of the last display is (easily) bounded below by $e^{-nA'\delta_n^2}$ for an appropriate (small enough) constant $A'$. It remains to lower bound the second factor. This can be done exactly as in step ii) of the proof of Theorem 2.2.2 in \cite{Nickl23} (see equation (2.26) and onward).

Finally, to prove (iii), we note that in view of Lemma~\ref{lemma: lipschitz} a covering of $\Theta\times H_n$ by  $d_T$-balls of radius $\overline{m}\delta_n$ can be obtained by combining sets from separate coverings of  $\Theta$  and $H_n$ by balls of radius $\overline{m}/2C(M)\delta_n$ with respect to the Euclidean distance and the $(H^2(\X))^*$-distance, respectively. The total number of sets is then the product of the two covering numbers, and it follows that
\[\log N(\overline{m}\delta_n,\Theta\times H_n, d_T)\leq \frac {4\text{diam}(\Theta)C(M)}{\overline{m}\delta_n}
+ \log N\Bigl(\frac{\overline{m}\delta_n}{2C(M)}, H_n, \|\cdot\|_{(H^2(\X))^*}\Bigr).\]
For $\overline{m}=\overline{m}(M)$ large enough, the first term is clearly an order smaller than $n\delta_n^2$,
while the second term is dealt with in the same way as in the end of the proof of Theorem 2.2.2 in \cite{Nickl23}.

\subsection{Proof of Lemma \ref{lemma: stability}}
We show below that there exists a constant $A(R)$ such that for $(\theta,F)\in \Theta\times B_{H^\beta}(R)$,
\begin{align}\label{eq: stab_neighborhood}
	|\theta-\theta_0| + \|F-F_0\|_2\leq A(R)\, \|K_{\theta}F-K_{\theta_0}F_0\|_{H^{2,1}}.
\end{align} 
The interpolation inequality \eqref{EqInterpolation} with $\nu=\beta/(2+\beta)$ gives
\begin{align*}
	\|K_\theta F - K_{\theta_0}F_0\|_{H^{2,1}}\lle \|K_\theta F-K_{\theta_0}F_0\|_2^{\beta/(2+\beta)} \|K_\theta F - K_{\theta_0}F_0 \| ^{2/(2+\beta)}_{H^{2+\beta, 1+\beta/2}}.
\end{align*}
Hence the lemma follows if we show that the $H^{2+\beta,1+\beta/2}(\Q)$-norm of $K_\theta F - K_{\theta_0}F_0$ is bounded by a constant,
which may depend on $R$, uniformly in $F\in B_{H^{\beta}}(R)$. This is true since by  Theorem \ref{thm: existence}, we have the bound
\begin{align*}
	\|K_{\theta}F- K_{\theta_0} F_0\|_{H^{2+\beta,1+\beta/2}}&\le \|K_{\theta_0}F_0\|_{H^{2+\beta,1+\beta/2}} + \|K_{\theta}F\|_{H^{2+\beta, 1+\beta/2}} \\
	&\lle 1+\|\Phi(F)\|_{H^\beta}^{1+\beta/2} \\
	&\lle 1+\|F\|_{H^\beta}^{\beta(1+\beta/2)}\\
	&\lle 1 + R^{\beta^2/2+\beta},
\end{align*}
where we used \eqref{eq: estimateSolution} for the second inequality and  (6.2) of Lemma 29 in \cite{nicklVdG} for the before last inequality. This concludes the proof of the lemma, except that it remains to prove 
\eqref{eq: stab_neighborhood}.

\subsubsection{Proof of \eqref{eq: stab_neighborhood}}

Below we prove that there exist constants $C_1(R)$, $C_3(R)$, $\mu(R)$ and $\nu>0$ such that for all 
$\theta\in\Theta$ and $F\in B_{H^{\beta}}(R)$ with $|\theta-\theta_0|+\|F-F_0\|_2<\mu(R)$:
\begin{enumerate}[label=(\roman*),leftmargin=*, itemsep=0.5ex, before={\everymath{\displaystyle}}]
	\item $\|K_\theta F- K_{\theta_0}F_0\|_{H^{2,1}}\geq C_1(R)\|F-F_0\|_2 - |\theta-\theta_0|$,
	\item $\|K_{\theta}F-K_{\theta_0}F_0\|_{H^{2,1}} \geq \tilde i_{\theta_0,F_0}^{1/2} |\theta-\theta_0|/2 - C_3(R)\|F-F_0\|^{1+\nu}_2$, 
\end{enumerate}
Using (i) once and (ii) $N\in\mathbb{N}$ times and adding the inequalities we see that, whenever $|\theta-\theta_0|+\|F-F_0\|_2<\mu(R)$,
\[(N+1)\|K_{\theta}F-K_{\theta_0}F_0\|_{H^{2,1}} 
\geq (N\tilde i_{\theta_0,F_0}^{1/2}/2-1)|\theta-\theta_0| + \bigl(C_1 - NC_3\|F-F_0\|_2^{\nu}\bigr)\|F-F_0\|_2.\]
We can choose $N$ large enough so that the coefficient of $|\theta-\theta_0|$ is positive.
Next for this given $N$ the coefficient of  $\|F-F_0\|_2$ is positive for sufficiently small $\|F-F_0\|_2$. 
It follows that \eqref{eq: stab_neighborhood} is valid for every $(\theta,F)\in\Theta\times B_{H^\beta(R)}$ with
sufficiently small $|\theta-\theta_0|+\|F-F_0\|_2$. 

Lemma~\ref{LemmaInversion} extends this to all $(\theta,F)\in\Theta\times B_{H^\beta(R)}$. Indeed, suppose there
were $(\theta,F)$ with $|\theta-\theta_0|+\|F-F_0\|_2\ge \mu_0$ for some $\mu_0>0$ such that 
$\bigl(|\theta-\theta_0| + \|F-F_0\|_2\bigr)/ \|K_{\theta}F-K_{\theta_0}F_0\|_{H^{2,1}}\to \infty$.  
Since $|\theta-\theta_0| + \|F-F_0\|_2$ is bounded uniformly in $(\theta,F)\in \Theta\times B_{H^\beta}(R)$,
it follows that $\|K_{\theta}F-K_{\theta_0}F_0\|_{H^{2,1}}\to 0$. Since $u_0$ does not lie in an eigenspace of $\SSS_{\theta_0,F_0}$, item (i) of Theorem \ref{thm:information} implies that $(\theta,F)\mapsto K_\theta F$ is injective at $(\theta_0, F_0)$ in view of Lemma \ref{lemma: identifiability}. By Lemma~\ref{LemmaInversion}, the convergence $\|K_{\theta}F-K_{\theta_0}F_0\|_{H^{2,1}}\to 0$ is therefore only possible if $|\theta-\theta_0|+\|F-F_0\|_2\to 0$, contradicting that this quantity is at least $\mu_0$. 

Finally we prove (i) and (ii).

\textit{Proof of (i).}  
Since $u_{\theta,F} \ (= K_\theta F) $ and $u_{\theta,F_0}$ are positive, we can recycle the proof of 
Proposition~10 in \cite{kekkonen} to obtain that for some constant $c$
\[\|\Phi(F)-\Phi(F_0)\|_2 \lle e^{c\|\Phi(F)\vee \Phi(F_0)\|_\infty}\|K_\theta F - K_\theta F_0\|_{H^{2,1}}.\] 
Here the multiplicative constant is uniformly bounded for $F\in B_{H^{\beta}}(R)$, since
$\|\Phi(F)\|_\infty \lle 1+\|F\|_\infty \lle 1+\|F\|_{H^\beta}$.
Using the triangle inequality and item (iii) of Lemma~\ref{lemma: bounds}, we infer the existence of a positive constant $D(R)$ such that,
\begin{align*}
	\|\Phi(F)-\Phi(F_0)\|_2&\lesssim D(R)\bigl(\|K_\theta F - K_{\theta_0} F_0\|_{H^{2,1}} + \|K_{\theta_0} F_0 - K_{\theta}F_0\|_{H^{2,1}} \bigr)\\
	&\leq D_1(R)\bigl(\|K_\theta F - K_{\theta_0} F_0\|_{H^{2,1}} + |\theta-\theta_0| \bigr).
\end{align*}
By Sobolev embedding and the interpolation inequality \eqref{EqInterpolation},
we have $\|F-F_0\|_\infty\lle \|F-F_0\|_{H^{\beta'}}\lle \|F-F_0\|_2^\nu \|F-F_0\|_{H^\beta}^{1-\nu}$,
for any $\beta'\in (d/2,\beta)$ and $\nu = 1-\beta'/\beta$. 
It follows that $\|F-F_0\|_\infty\to 0$ if $\|F-F_0\|_2\to 0$ and $F\in B_{H^\beta}(R)$.
Hence because $F_0$ is bounded away from 0 and $\infty$, so is $F$ for sufficiently small $\|F-F_0\|_2$.
Because $\Phi'$ is bounded away from zero on compact intervals, it follows that
$\|F-F_0\|_2\lesssim \|\Phi(F)-\Phi(F_0)\|_2$ for sufficiently small $\|F-F_0\|_2$.
Together with the preceding display this gives (i).

\textit{Proof of (ii).} We write the difference $K_\theta F - K_{\theta_0}F_0$ as
\[K_\theta F - K_{\theta_0}F_0 = (\theta-\theta_0)\dot K_{\theta_0}F_0 + I_{\theta_0,F_0}(F-F_0) + R(\theta,F),\]
where the remainder can be decomposed as
\begin{align*}
	R(\theta,F)&=K_\theta F- K_{\theta_0}F - (\theta - \theta_0)\dot K_{\theta_0} F
	+ (\theta-\theta_0)(\dot K_{\theta_0}F - \dot K_{\theta_0}F_0) \\
	&\qquad\qquad+ K_{\theta_0}F-K_{\theta_0}F_0 - I_{\theta_0,F_0}(F-F_0).
\end{align*}
By items (v), (iv) and (vi) of Lemma~\ref{lemma: bounds}, there exists a constant $C_2=C_2(R)$ for which,
\[\|R(\theta,F)\|_{H^{2,1}}\leq C_2\left(|\theta-\theta_0|^2 + |\theta-\theta_0|\|F-F_0\|_2 + \|F-F_0\|_2\|F-F_0\|_\infty \right).\] 
Together with  the triangle inequality this implies that $\|K_\theta F-K_{\theta_0}F_0\|_{H^{2,1}}$ is lower bounded
by, whenever $|\theta-\theta_0|+\|F-F_0\|<\mu$, 
\begin{align*}
	&\bigl\|(\theta-\theta_0)\dot K_{\theta_0}F_0 + I_{\theta_0,F_0}(F-F_0)\bigr\|_{H^{2,1}} - C_2|\theta-\theta_0|\mu - C_2\|F-F_0\|_2\|F-F_0\|_\infty \nonumber\\
	&\quad\geq |\theta-\theta_0|\left(\left\|\dot K_{\theta_0}F_0 + I_{\theta_0,F_0}\left[\dfrac{F-F_0}{\theta-\theta_0}\right]\right\|_{2} - C_2\mu\right) - C_2\|F-F_0\|_2\|F-F_0\|_\infty.
\end{align*}
The efficient Fisher information $\tilde i_{\theta_0,F_0}$ is the infimum of 
$\|\dot K_{\theta_0}F_0 + I_{\theta_0,F_0}h\|^2_2$ over $h\in H^\beta(\X)$. It follows that the coefficient of $|\theta-\theta_0|$ in the last display is bounded below by $\tilde i_{\theta_0,F_0}^{1/2} - C_2\mu\ge \tilde i_{\theta_0,F_0}^{1/2}/2$, for sufficiently small $\mu=\mu(R)$.
Furthermore, as seen under (i), by Sobolev embedding and interpolation 
we have $\|F-F_0\|_\infty\lle \|F-F_0\|_2^\nu \|F-F_0\|_{H^\beta}^{1-\nu}$,
for some $\nu\in(0,1)$, which is further bounded above by
$\|F-F_0\|_2^\nu (1+R)^{1-\nu}$ for $F\in B_{H^\beta}(R)$. Thus we can replace the norm $\|F-F_0\|_\infty$
in the preceding display by a constant depending on $R$ times $\|F-F_0\|_2^\nu$.

	\begin{lemma}
		\label{LemmaInversion}
		Fix $R>0$ and set $B_{H^\beta}(R)=\{F\in H^\beta(\X):= \|F\|_{H^\beta(\X)}\le R\}$.
		Assume that the map $(\theta, F)\mapsto K_\theta F$ from $\Theta\times L^2(\X)$ is  injective at $(\theta_0, F_0)$ in the sense
		that $K_\theta F=K_{\theta_0}F_0$ implies $(\theta, F)=(\theta_0,F_0)$ and that $K_{\theta_0}F_0$ is bounded away from zero.
		Then $K_\theta F\to K_{\theta_0}F_0$ in $H^{2,1}(\Q)$ for $(\theta,F)\in \Theta\times B_{H^\beta}(R)$  implies that
		$(\theta,F)\to (\theta_0,F_0)$ in $\Theta\times L^2(\X)$.
	\end{lemma}
	
	\begin{proof}
		Take any sequence $(\theta_m,F_m)\in \Theta\times B_{H^\beta}(R)$  with $\theta_m\to \theta_1$ for some $\theta_1$
		and $K_{\theta_m}F_m\to u:=K_{\theta_0}F_0$ in $H^{2,1}(\Q)$. 
		%Because $\|K_{\theta_m}F_m-K_{\theta_1}F_m\|_{H^{2,1}}\lesssim (1+R)^2 |\theta_m-\theta_1|$, by Lemma~\ref{lemma: bounds}, we have that $K_{\theta_1}F_m\to u_0$ in $H^{2,1}(\Q)$. 
		This implies that $K_{\theta_m}F_m\to u$ also for the uniform norm and hence
		$K_{\theta_m}F_m$ is bounded away from zero, for sufficiently large $m$. Since $u_{\theta_m,F_m}$ satisfies \eqref{eq:PDE},
		we can solve $F_m=-(\partial_t -\theta_m\Delta_x K_{\theta_m}F_m/2)/K_{\theta_m}F_m$, which converges
		in $L_2(\X)$ to $F_1:=-(\partial_t -\theta_m\Delta_x K_{\theta_m}u/2)/u$. Then items (iii) and (i) of Lemma~\ref{lemma: bounds}
		give that $K_{\theta_m}F_m\to K_{\theta_1}F_1$ in $L_2(\Q)$. Hence $K_{\theta_1}F_1=K_{\theta_0}F_0$ and $(\theta_1,F_1)=(\theta_0,F_0)$
		by the assumed injectivity. 
		
		Because $\Theta$ is compact, any subsequence of a sequence $(\theta_m,F_m)$ possesses a further subsequence along which
		$\theta_m$ converges to a limit. Applying the preceding argument to this subsequence we see that $\theta_0$ is the only
		possible limit point and the proof is complete.
	\end{proof}

\subsection{Proof of Lemma \ref{lemma: lipschitz}}
By the triangle inequality 
$$\|K_{\theta_1}F_1-K_{\theta_2}F_2\|_2\leq \|(K_{\theta_1}-K_{\theta_2})F_1\|_2 + \|K_{\theta_2}F_1-K_{\theta_2}F_2\|_2.$$
By item (iii) of Lemma~\ref{lemma: bounds}, the first term on the right is bounded above by
$D_3(1+\|F_1\|_\infty)^2|\theta_1-\theta_2|$. 
Because $\beta>2+d/2>d/2$, the space  $H^\beta(\X)$  is continuously embedded in of continuous functions $C(\X)$
and hence $\|F_1\|_\infty\lesssim\|F_1\|_{H^\beta}$
remains uniformly bounded on $B_{H^\beta}(R)$.
By a slight adaptation of Proposition~9 in \cite{kekkonen}, the second term on the right is bounded above by a multiple of
$(1+\|F_1\|^4_{C^2(\X)} \vee \|F_2\|^4_{C^2(\X)})\|F_1-F_2\|_{(H^2(\X))^*}$. 
Since $\beta>2+d/2$, the space $H^\beta(\X)$ is continuously embedded in $C^2(\X)$, and hence again the constant remains uniformly
bounded over balls of radius $R$.

\section{Proofs of the LAN Expansion and the Information Theorem (Section \ref{SectionLAN})}

\label{sec:proofInfo}

\subsection{Proof of Lemma \ref{lemma: devs}}

If it exists, the derivative $\dot K_\theta F$ of $\tau\mapsto K_\tau F$ at $\theta$ is the limit of the ratio $(K_{\theta+s}F - K_\theta F)/s$ as $s\to 0$. Observe that $K_{\theta+s} F - K_\theta F$ vanishes on the parabolic boundary. We have that,
\begin{align*}
	\eL_{\theta,F} (K_{\theta+s} F - K_\theta F) &= \eL_{\theta, F}K_{\theta+s}F - 0\\
	&= (\eL_{\theta,F} - \eL_{\theta+s,F})K_{\theta+s} F= -\frac{s}{2}\Delta_x K_{\theta+s}F.
\end{align*}
As we already mentioned, $\eL_{\theta,F}:H_{B,0}^{2,1}(\Q)\to L^2(\Q)$ is an isomorphism with continuous inverse $\eL_{\theta,F}^{-1}$. We thus have,
\begin{align*}
	\frac{K_{\theta+s} F - K_\theta F}{s} = \eL_{\theta, F}^{-1}\Bigl(-\frac{\Delta_x K_{\theta+s}F}{2}\Bigr).
\end{align*}
Using Lemma \ref{lemma:firstEstimate} and the fact that $\Delta_x:H^{2,1}_{B,0}(\Q)\to L^2(\Q)$ is continuous, we further get that
\begin{align*}
	\left\| \eL_{\theta,F}^{-1}\left(\frac{\Delta_x(K_\theta F - K_{\theta+s}F)}{2}\right)\right\|_2\leq \frac{C}{2}\|\Delta_x(K_\theta F - K_{\theta+s}F)\|_2\lle \|K_{\theta}F-K_{\theta+s}F\|_{H^{2,1}},
\end{align*}
and the right hand side goes to 0 as $s\downarrow 0$ by virtue of item (iii) in Lemma \ref{lemma: bounds}. Putting everything together thus yields

\begin{align*}
	\frac{K_{\theta+s} F - K_\theta F}{s} = \eL_{\theta, F}^{-1}\Bigl(-\frac{\Delta_x K_{\theta+s}F}{2}\Bigr)\to \eL_{\theta,F}^{-1}\Bigl(-\frac{\Delta_x K_\theta F}{2}\Bigr), \quad\text{ as } s\to 0.
\end{align*}

We now move on to the proof of the second part of the lemma. By assumption $\Phi$ satisfies, for all $x,s \in\mathbb{R}$,
\begin{align}\label{eq: PhiCondition}
	\left| \Phi(x+s) - \Phi(x)\right|\lesssim |s|, \quad \left|\Phi(x+s)-\Phi(x)-s\Phi'(x)\right|\lesssim |s|^2.
\end{align}
Furthermore, as seen in Section \ref{sec: classicalSol}, since $\Theta \subset (0,\infty)$ is compact and since  $\Phi(F) > 0$ for all $F\in H$, we also have that 
\begin{align}\label{eq: uniformBound}
	\sup_{(\theta,F)\in \Theta\times H} \|K_\theta F\|_\infty \leq \|u_0\|_\infty =: U
\end{align}
Observe then that $K_\theta (F+h) - K_\theta F$ vanishes on the parabolic boundary, and so does $I_{\theta,F}h$ (as it is in the image of $\eL_{\theta, F}^{-1}$). It follows that their difference vanishes on the parabolic boundary as well. Now,
\begin{align*}
	&\eL_{\theta, F+h}(K_\theta (F+h) - K_\theta F - I_{\theta,F}h) = \eL_{\theta, F+h}(0 - K_\theta F - I_{\theta,F}h)\\
	&\qquad= (\eL_{\theta, F} - \eL_{\theta, F+h})(K_\theta F + I_{\theta,F}h) - h\Phi'(F) K_\theta F\\
	&\qquad= \left(\Phi(F+h) - \Phi(F) - h\Phi'(F)\right)K_\theta F + (\Phi(F+h)-\Phi(F))I_{\theta,F}h.
\end{align*}
It follows that $\|K_\theta (F+h) - K_\theta F - I_{\theta,F}h\|_2$ is equal to
\begin{align*}
	\Bigl\|&\eL_{\theta, F+h}^{-1}\bigl[\bigl(\Phi(F+h) - \Phi(F) - h\Phi'(F)\bigr)K_\theta F 
	+ (\Phi(F+h)-\Phi(F))I_{\theta,F}h\bigr]\Bigr\|_2\\
	&\leq C\left\| \bigl(\Phi(F+h) - \Phi(F) - h\Phi'(F)\bigr)K_\theta F
	+ \bigl(\Phi(F+h)-\Phi(F)\bigr)I_{\theta,F}h\right\|_2 \quad\text{ (by \eqref{eq: LipsEst1})}\\
	&\leq C\|K_\theta F\|_\infty\|h\|_\infty \|h\|_2 + C \|h\|_\infty\|I_{\theta,F}h\|_2 \qquad\text{ (by \eqref{eq: PhiCondition})}\\
	&=  C\|K_\theta F\|_\infty\|h\|_\infty \|h\|_2 + C \|h\|_\infty\|\eL_{\theta,F}^{-1} (h\Phi'(F)K_\theta F)\|_2\\
	&\leq CU\left(\|h\|_\infty\|h\|_2 + \|h\|_\infty \|h\|_2\|\Phi'(F)\|_\infty\right) \qquad\text{ (by \eqref{eq: LipsEst1} and \eqref{eq: uniformBound}) }\\
	&\lesssim \|h\|_\infty\|h\|_2= o(\|h\|_\infty),
\end{align*}
where we repeatedly used the fact that $\|uv\|_2\leq \|u\|_\infty\|v\|_2$, for $u\in L^\infty(\Q), v\in L^2(\Q)$.

We finally show continuity of the map $\theta\mapsto I_{\theta,F}h$. Observe that 
\[\eL_{\theta+s, F} (I_{\theta+s,F}h-I_{\theta,F}h) = \Phi'(F)h(K_{\theta+s}F-K_\theta F) + (\eL_{\theta,F}[I_{\theta,F}h] - \eL_{\theta+s, F}[I_{\theta,F}h]).\]
By \eqref{eq: LipsEst1} in Lemma \ref{lemma:firstEstimate}, we can thus bound $\|I_{\theta+s,F}h-I_{\theta,F}h\|_2$ by a constant multiple of the $L^2(\Q)$-norm of the two terms on the right hand side of the above equality. The $L^2$-norm of the first term goes to 0 as $s\downarrow 0$ by virtue of item (iii) in Lemma \ref{lemma: bounds}. Besides, item (vii) of Lemma \ref{lemma: bounds} gives us that $I_{\theta,F}h\in H^{2,1}(\Q)$. The $L^2$-norm of the second term then goes to 0 as $s\downarrow 0$ because $\theta\mapsto \eL_{\theta,F}u$ is continuous for all $u\in H^{2,1}(\Q)$, since $\|\eL_{\tau,F}u-\eL_{\theta,F}u\|_2 = \frac{1}{2}|\theta-\tau|\|\Delta_x u\|_2$. 

\subsection{Proof of Lemma \ref{lemma: LAN} (LAN Expansion)}

By Lemma~\ref{lemma: devs}, for any $a\in\mathbb{R}$, as $s\downarrow 0$,
\begin{align}\label{eq: limitDev}
	\begin{split}
		\frac{K_{\theta+sa} (F+saG) - K_\theta F}{s} &= \frac{K_{\theta+sa}F-K_\theta F}{s}+ \frac{K_{\theta+sa}(F+saG) - K_{\theta+sa} F}{s} \\
		&\to a\dot K_\theta F + aI_{\theta,F} G,
	\end{split}
\end{align}
where the limit is in $L^2(\Q)$. For the first term, this is a direct consequence of the first part of Lemma \ref{lemma: devs}, while for the second term this follows from the proof of its second part combined with continuity of $\theta\mapsto I_{\theta,F}G$ since
\begin{align*}
	&\left\| \frac{K_{\theta+sa}(F+saG) - K_{\theta+sa} F}{s} - aI_{\theta,F}G\right\|_2\\
	&= \left\| \frac{K_{\theta+sa}(F+saG) - K_{\theta+sa} F}{s} - aI_{\theta+sa,F}G+aI_{\theta+sa,F}G-aI_{\theta,F}G\right\|_2\\
	&\lle sa^2\|G\|_\infty\|G\|_2+a\|I_{\theta+sa,F}G-I_{\theta,F}G\|_2\\&\to 0 \;\;\; \text{as }s\to 0.
\end{align*}

Applying formula \eqref{eq: loglikelihood} twice then yields for $(\tau,G)\in\Theta\times H$,
with $X^n=K_\theta F+n^{-1/2}\dot W$,
\begin{align*}
	\begin{split}
		\log{\frac{dP_{\tau,G}^n}{dP_{\theta,F}^n}} =\sqrt{n}\ip{K_\tau G - K_\theta F}{\dot W}_2 - \frac{n}{2}\|K_\tau G-K_\theta F\|_2^2.
	\end{split}
\end{align*}
It follows that the left hand side of \eqref{eq: LANexp} is equal to,
\begin{align*}
	\Bigip{\frac{K_{\theta+s/\sqrt{n}} (F+sG/\sqrt{n}) - K_\theta F}{1/\sqrt{n}}}{\dot W}_2 - \frac{1}{2}\Bigl\| \frac{K_{\theta+s/\sqrt{n}} (F+sG/\sqrt{n}) - K_\theta F}{1/\sqrt{n}}\Bigr\|_2^2
\end{align*}
Taking $(a,s)$ in \eqref{eq: limitDev} equal to  the present $(s,1/\sqrt{n})$ shows that the limit of the above display as $n\to\infty$ is equal to the right hand side of \eqref{eq: LANexp}.

Minimising the information $G\mapsto \|\dot K_\theta F + I_{\theta, F}G\|^2$ is to minimize the $L^2$-distance of $-\dot K_\theta F$ to
the (closure of the) range of $I_{\theta,F}$. The minimiser is the orthogonal projection  of $-\dot K_\theta F$ onto the latter space.
If there exists $h\in H$ such that $I_{\theta, F}^*\dot K_\theta F = I_{\theta,F}^*I_{\theta,F}h$, then this projection 
takes the form $  I_{\theta,F}h$, since in this case, for every $G$,
$$\langle \dot K_\theta F-I_{\theta,F}h, I_{\theta,F}G\rangle_{L^2(\Q)}=
\langle I_{\theta,F}^* \dot K_\theta F-I_{\theta,F}h, G\rangle_{L^2(\Q)}=\langle 0,G\rangle_{L^2(\Q)}=0.$$
It follows that the efficient Fisher information is $\tilde i_{\theta,F}$ as in the statement of the lemma.

Provided $\dot K_\theta F \not= 0$, the efficient Fisher information is zero only if $\dot K_{\theta}F=I_{\theta,F}\gamma_{\theta,F}$.
From the expressions for $\dot K_\theta F$ and $I_{\theta,F}\gamma_{\theta,F}$ in Lemma~\ref{lemma: devs},
we see that this equality is equivalent to
\[ \gamma_{\theta,F} = \frac{\Delta_x K_\theta F}{2\Phi'(F) K_\theta F}
=\frac{1}{\theta} \left( \frac{\partial_t K_\theta F}{\Phi'(F)K_\theta F} + \frac{\Phi(F)}{\Phi'(F)}\right), \]
since $\eL_{\theta,F}K_{\theta}F=0$.
Because  $\gamma_{\theta,F}$ and $F$ are functions of only the space variable and not time, it follows that
the function $\partial_t K_\theta F / K_\theta F = \partial_t \log K_\theta F$ does not depend on time.
This is excluded by the assumption in the lemma that $\partial_t^2 \log{K_\theta F} \neq 0$. Finally, it is also impossible that $\dot K_\theta F = 0$ under this assumption. If that is the case, i.e. if $\eL_{\theta,F}^{-1}(\Delta_x K_\theta F) = 0$, then $\Delta_x K_\theta F=0$, and since $\eL_{\theta,F}K_\theta F = 0$, this further implies that 
\[\del_t K_\theta F = - \Phi(F)K_\theta F \implies \del_t\log{K_\theta F} = -\Phi(F),\]
which would give again that $\del_t\log{K_\theta F}$ does not depend on time, contradicting the assumption.

\subsection{Proof of Theorem \ref{thm:information}}

\subsubsection{Proof of (i)}

Suppose by contradiction that $\del_t^2\log K_{\theta}F =0$. Then for some functions  $A:\X\to\RR$ and $B:\X\to\RR$, we can write $K_{\theta}F(x,t) = B(x)e^{tA(x)}$. It follows that 
\begin{align*}
	\Delta K_\theta F &= \nabla\cdot(e^{At}(\nabla B + tB\nabla A))\\
	&=e^{At}\Bigl( \Delta B + 2t\nabla B\cdot\nabla A  + tB\Delta A  + t^2B|\nabla A|^2 \Bigr)\\
	&= K_\theta F \Bigl(  \frac{\Delta B}{B} + 2t \frac{\nabla B\cdot\nabla A}{B} + t\Delta A + t^2|\nabla A|^2\Bigr),
\end{align*}
where we have used that since  $K_\theta F>0$ (by assumption on $u_0$ being positive), then also $B(x)>0$. Using the previous display and the fact that $\del_t K_\theta F = AK_\theta F$, we can thus divide by $K_\theta F$ on both sides of the equation $\eL_{\theta, F}K_\theta F =0$ to obtain that identically on $\Q$, 
\[A - \frac{\theta}{2}\Bigl( \frac{\Delta B}{B} + 2t \frac{\nabla B\cdot\nabla A}{B} + t\Delta A + t^2|\nabla A|^2\Bigr) + \Phi(F) = 0.\] 
Since $A$ and $\Phi(F)$ depend exclusively on the space variable, we must have that the $t$-polynomial in parentheses is constant (w.r.t time). In particular, the $t^2$-coefficient must be zero. This implies that $\nabla A = 0\implies A=\lambda,$ for some $\lambda\in\RR$. It follows that $K_\theta F(x,t) = B(x)e^{\lambda t}$, and since $u_0(x)=K_\theta F(x,0)=B(x)$, we can then write $K_\theta F(x,t)=u_0(x)e^{\lambda t}$. The equation $\eL_{\theta,F}K_\theta F=0$ is then equivalent to 
\[\del_t (u_0e^{\lambda t}) = \SSS_{\theta,F} (u_0e^{\lambda t}) \iff e^{\lambda t}\lambda u_0  = e^{\lambda t}\SSS_{\theta,F}(u_0)\iff\lambda u_0 = \SSS_{\theta,F}(u_0). \]
This precisely means that $u_0$ lies in an eigenspace of $\SSS_{\theta,F}$, which contradicts the assumption and concludes the proof.

\subsubsection{Proof of (ii)}

We have seen in Section \ref{SectionLAN} that $\eL_{\theta,F}:H_{B,0}^{2,1}(\Q)\to L^2(\Q)$ is an isomorphism. Its adjoint $\eL_{\theta,F}^*:H_{C,0}^{2,1}(\Q)\to L^2(\Q)$ is $\eL_{\theta,F}^*= -\partial_t - \frac{\theta}{2}\Delta_x + \Phi(F)$, and we have the identity $(\eL_{\theta,F}^{-1})^* = (\eL_{\theta,F}^*)^{-1}$. It is then straightforward to check that $I_{\theta,F}^*: L^2(\Q)\to L^2(\X)$
has the following expression, for $u\in L^2(\Q)$:
\begin{align*}
	I_{\theta,F}^*u(\cdot) = \Phi'(F)\int_0^T K_{\theta}F(\cdot,t)\,\bigl((\eL_{\theta,F}^*)^{-1}u\bigr)(\cdot, t)\,dt.
\end{align*}

We begin by showing that $I^*_{\theta,F}\dot K_\theta F\in H^\xi(\X)$. We first recall that $\dot K_\theta F = \eL_{\theta,F}^{-1}(\Delta_x K_\theta F/2)$. By assumption, $\xi>2+d/2$. Hence, by Theorem \ref{thm: existence}, we have that $K_\theta F\in H^{2+\xi, 1+\xi/2}(\Q)$. Furthermore, by Proposition 2.3 in \cite{lionsVol2} applied with $r=2+\xi, s=1+\xi/2, j= 2$ and $k=0$, we have that $\Delta_x K_\theta F\in H^{\xi, \xi/2}(\Q)$. Item (iii) of Lemma \ref{lemma: refinedLipschitz} then implies that $\dot K_\theta F\in H^{2+\xi, 1+\xi/2}(\Q)$. It follows that $I^*_{\theta,F}\dot K_\theta F\in H^\xi(\X)$ provided we can show that $I^*_{\theta,F}$ maps $H^{2+\xi,1+\xi/2}(\Q)$ into $H^\xi(\X)$. To this end, we observe that the statements of Lemma \ref{lemma: refinedLipschitz} also hold when replacing $\eL_{\theta,F}^{-1}$ with $(\eL_{\theta,F}^*)^{-1}$ (it simply suffices to replace $L_\theta$ in its proof by $L^*_\theta:=\frac{\theta}{2}\Delta_x + \del_t$, and the result follows). We thus can apply item (iii) again to show that if $u\in H^{2+\xi, 1+\xi/2}(\Q)\subset H^{\xi,\xi/2}(\Q)$, then certainly $(\eL_{\theta,F}^*)^{-1}u\in H^{2+\xi,1+\xi/2}(\Q)$. It follows that the integrand in the expression of $I^*_{\theta,F}u$ in the last display above is in $H^{2+\xi, 1+\xi/2}(\Q)$, so the integral (as a function on $\X$) is certainly in $H^\xi(\X)$. Finally, since  $\xi>d/2$, we have that $\Phi'(F)\in H^\xi(\X)$, and thus also (by \eqref{EqSobolevProductOne}) that $I^*_{\theta,F}\dot K_\theta F\in H^\xi(\X)$. 
%The reason that $\Phi'(F)\in H^\xi(F)$ is that  $F\in H^\xi(\X)\cap L^\infty(\X)$  (because \xi>d/2) and $\Phi'$ is smooth with bounded devs. So Theorem 2.87 in "Fourier Analysis and Nonlinear Partial Differential Equations", Bahouri et.al, 2011, applies.

We now prove that $\II:H^\alpha(\X)\to H^\alpha(\X)$ is an isomorphism. We will proceed by a Fredholm argument; first showing that $\II$ is injective, and then showing that it can be written as $\II = \JJ + R$, for $\JJ:H^\alpha(\X)\to H^\alpha(\X)$ an isomorphism and $R:H^\alpha(\X)\to H^{(\xi-4)\wedge (\alpha+1)}(\X)$ compact. The result will then follow by the Fredholm alternative (Theorem VI.6 in \cite{brezis1983}). 
\smallskip

\textit{Injectivity.} We will show that $I^*_{\theta,F}I_{\theta,F}:H^{\alpha}(\X)\to H^{\alpha+4}(\X)$ is injective. Injectivity of $\SSS_{\theta,F}:H^{s}(\X)\to H^{s+2}(\X)$ for $s\geq 1$ can then be applied twice (with $s=\alpha+4$ first and with $s=\alpha+2$ second) to conclude that $\II:H^\alpha(\X)\to H^\alpha(\X)$ is indeed injective. 

Recall that $I_{\theta,F}h = \eL_{\theta,F}^{-1}(h\Phi'(F)K_\theta F)\in H^{2,1}_{B,0}(\Q)$. Assume now that $I_{\theta,F}h = 0$ for some $h\in H^\alpha(\X)$. Using that $\eL_{\theta,F}:H_{B,0}^{2,1}(\Q)\to L^2(\Q)$ is an isomorphism, there exists some constant $C_{\theta,F}>0$ such that $\|\eL_{\theta,F}u\|_2\leq C_{\theta,F}\|u\|_{H^{2,1}(\Q)}$ for $u\in H^{2,1}_{B,0}(\Q)$. Therefore, we have
\[\|h\Phi'(F)K_\theta F\|_2 = \Bigl\|\eL_{\theta,F}\Bigl[I_{\theta,F}h\Bigr]\Bigr\|_2\leq C_{\theta,F}\|I_{\theta,F}h\|_{H^{2,1}} = 0. \]
By assumption $\Phi'(F)K_\theta F >0$, and thus we obtain from the last display that $h=0$. This shows injectivity of $I_{\theta,F}$. Let then $h\in H^\alpha(\X)$ be such that $I_{\theta,F}^*I_{\theta,F}h=0$. Then  $\ip{I^*_{\theta,F}I_{\theta,F}h}{h}_2 = 0$, which is equivalent to $\|I_{\theta,F}h\|^2_2=0$. Injectivity of $I_{\theta,F}$ then implies that $h=0$, which shows that $I^*_{\theta,F}I_{\theta,F}$ is injective on $H^\alpha(\X)$. 

It remains to show that it maps $H^\alpha(\X)$ into $H^{\alpha+4}(\X)$. To this end, observe that since $\xi>d/2,\ \Phi'(F)\in H^\xi(\X)$ and $K_\theta F\in H^{2+\xi, 1+\xi/2}(\Q)$, we have by \eqref{EqSobolevProductOne} that $h\Phi'(F)K_\theta F\in H^{\xi,\xi/2}(\Q)\subset H^{\alpha, \alpha/2}(\Q)$. Because $F\in H^\xi(\X)\subset H^\alpha(\X)$, we can apply \eqref{eq:refinedH} in Lemma \ref{lemma: refinedLipschitz} to obtain that $I_{\theta,F}h\in H^{2+\alpha, 1+\alpha/2}(\Q)$. Furthermore, since $F\in H^\xi(\X)\subset H^{\alpha+2}(\X)$, we can apply \eqref{eq:refinedH} a second time (this time with $(\eL_{\theta,F}^*)^{-1}$ which we explained was valid above) to obtain that $(\eL_{\theta,F}^*)^{-1}[I_{\theta,F}h]\in H^{4+\alpha, 2+\alpha/2}(\Q)$. As $K_\theta F\in H^{\xi+2, 1+\xi/2}(\Q)\subset H^{4+\alpha, 2+\alpha/2}(\Q)$, we apply \eqref{EqSobolevProductOne} again to show that $K_\theta F(\eL_{\theta,F}^*)^{-1}[I_{\theta,F}h]\in H^{4+\alpha, 2+\alpha/2}(\Q)$. Integrating this with respect to time then yields a function in $H^{\alpha+4}(\X)$, which multiplied by $\Phi'(F)\in H^{\xi}(\X)\subset H^{\alpha+4}(\X)$ yields a function in $H^{\alpha+4}(\X)$ again by \eqref{EqSobolevProductOne}. It follows that for $h\in H^\alpha(\X),$
\[I_{\theta,F}^*I_{\theta,F}h = \Phi'(F)\int_0^T K_{\theta}F(\cdot,t)\,\bigl((\eL_{\theta,F}^*)^{-1}\bigl[I_{\theta,F}h\bigr]\bigr)(\cdot, t)\,dt\in H^{\alpha+4}(\X).\]
\smallskip

\textit{Fredholm Decomposition.} Since $\theta$ and $F$ are fixed here, we will denote $\eL=\eL_{\theta,F}$ and $\SSS=\SSS_{\theta,F}$ to ease notation. For any $h\in H^\alpha(\X)$, we also define the following functions:
 \begin{align*}
	g&:=\Phi'(F)K_\theta F,\\
	H_h&:=(\eL)^{-1}[hg],\\
	G_h&:=(\eL^*)^{-1}\bigl[H_h].
\end{align*}
It follows that we have
\begin{align*}
	\II h = \SSS^2\Bigl[\int_0^Tg(t)G_h(t)dt\Bigr] 
	&= \int_0^T\SSS\Bigl[\SSS[g(t)G_h(t)]\Bigr]dt
\end{align*}
Note further that for two functions $a$ and $b$, we have 
\[\SSS[ab] = a\SSS[b] + \frac{\theta}{2}b\Delta a + \theta\nabla a\cdot\nabla b.\]
Successive application of this identity yields
\begin{align}
	\II h 	 &= \int_0^T \SSS\Bigl[g(t)\SSS[G_h(t)] + \frac{\theta}{2}G_h(t)\Delta g(t) + \theta \nabla g(t)\cdot\nabla G_h(t) \Bigr]dt\nonumber\\
	&= \int_0^Tg(t)\SSS^2[G_h(t)]dt + \frac{\theta^2}{4}\int_0^T G_h(t)\Delta^2[g(t)]dt + R_1 h,\label{eq:II} 
\end{align}
with 
\begin{align*}
	R_1h = \frac{\theta}{2}\int_0^T &2\Delta g(t) \SSS[G_h(t)] + 2\nabla \SSS[G_h(t)]\cdot \nabla g(t)\\
	&+\theta\nabla(\Delta g(t))\cdot \nabla G_h(t) + 2\SSS\Bigl[ \nabla g(t)\cdot\nabla G_h(t)\Bigr] dt 
\end{align*}
Now, the arguments used in the proof of injectivity of $\II$ give us that 
\begin{align*}
	g&\in H^{\xi, \xi/2}(\Q),\\
	G_h&\in H^{4+\alpha, 2+\alpha/2}(\Q).
\end{align*}
In combination with Proposition 2.3 in Chapter 4 of \cite{lionsVol2} repeatedly applied for $\nabla$ and $\Delta$, this yields to the following memberships: 
\begin{align*}
	\nabla g &\in H^{\xi-1,\xi/2-1/2}(\Q), \\ \Delta g &\in H^{\xi-2, \xi/2-1}(\Q), \\ \nabla(\Delta g)&\in H^{\xi-3, \xi/2 -3/2}(\Q),\\
	\nabla G_h &\in H^{\alpha+3,\alpha/2+3/2}(\Q),\\
	\Delta G_h&\in H^{\alpha+2, \alpha+1}(\Q),\\
	\nabla (\Delta G_h)&\in H^{\alpha+1,\alpha+1/2}(\Q).
\end{align*}
It follows using \eqref{EqSobolevProductOne} and that $\xi>\alpha+4$, that the integrand in the expression for $R_1h$ is in $H^{\alpha+1,\alpha/2+1/2}(\Q)$, from which we conclude that $R_1h\in H^{\alpha+1}(\X)$. Besides, again by Proposition 2.3 in Chapter 4 of \cite{lionsVol2}, we have that  $\Delta^2 g\in H^{\xi-4,\xi/2-2}(\Q)$. It follows that the second term in \eqref{eq:II} is in $H^{\xi-4}(\X)$ which is compactly embedded in $H^\alpha(\X)$ by assumption. We therefore have that $R_1'$ defined below
\[R'_1 h := \frac{\theta^2}{4}\int_0^T G_h(t)\Delta^2[g(t)]dt + R_1 h, \]
is a compact operator mapping  $H^{\alpha(\X)}$ into $H^{(\alpha+1)\wedge(\xi-4)}(\X)$. We now deal with the first term in \eqref{eq:II}. Observe firstly that
\begin{align*}
	g\SSS^2[G_h]&= -g\SSS[(\eL^* + \del_t)(\eL^*)^{-1}[H_h]]\\
	&= -g\SSS[H_h] - \del_t\SSS[G_h]\\
	&= g(\eL -\del_t)(\eL)^{-1}[hg] -\del_t \SSS[G_h]\\
	&=g^2h - g\del_t \left(H_h + \SSS[G_h]\right).
\end{align*}
Note then that since $G_h\in H^{2,1}_{C,0}(\Q)$ and $H_h\in H^{2,1}_{B,0}(\Q)$, we have $G_h(T)=H_h(0)=0$. Integrating the above expression using integration by parts therefore yields
\begin{align*}
	 h\int_0^{T}g^2(t)dt + g(0)\SSS[G_h](0) - g(T)H_h(T) + \int_0^T (\del_t g)\left(H_h + \SSS[G_h]\right)(t)dt.
\end{align*}
Letting $m:x\mapsto \int_0^Tg(x,t)^2dt$, we therefore obtain that $\II h = \JJ h + R_1'h + R_2h$ for
\begin{align*}
	\JJ h &= mh, \\
	R_2 h &= g(0)\SSS[G_h](0) - g(T)H_h(T) + \int_0^T (\del_t g)\left(H_h + \SSS[G_h]\right)(t)dt.
\end{align*}
Since $g\in H^{\xi,\xi/2}(\Q)$ is positive, we obtain that $m\in H^\xi(\X)\subset H^\alpha(\X)$ is also positive. As $\alpha>d/2,$ Sobolev embedding imply that $m$ is a continuous bounded function on $\X$. As $\X$ is compact, positivity implies that $m$ is also bounded away from zero. It follows that $\JJ:H^\alpha(\X)\to H^\alpha(\X)$ is  an isomorphism (with inverse $h\mapsto m^{-1}h$). %The reason that $m^{-1}\in H^\alpha(\X)$ is that  $m\in H^\alpha(\X)\cap L^\infty(\X)$  (because \alpha>d/2) and $x\mapsto x^{-1}$ is a smooth function with bounded devs on [\inf{m}_{x\in\X}, \infty). So Theorem 2.87 in "Fourier Analysis and Nonlinear Partial Differential Equations", Bahouri et.al, 2011, applies.
 Furthermore, in view of what preceded, we already know that both $H_h$ and $\SSS[G_h]$ are in $ H^{2+\alpha, 1+\alpha/2}(\Q)$. Proposition 2.3 in Chapter 4 of \cite{lionsVol2} also yields $\del_t g \in H^{\xi-2,\xi/2-1}(\Q)\subset H^{2+\alpha, 1+\alpha/2}(\Q)$. By \eqref{EqSobolevProductOne}, the integral in the expression for $R_2h$ as a function on $\X$ is thus in $H^{\alpha+2}(\X)$. Finally, by Theorem 4 in \S5.9.2 in \cite{EvansPDE}, we have that both $H_h(T)$ and $\SSS[G_h](0)$ are in $H^{\alpha+1}(\X)$. By the same result, both $g(0)$ and $g(T)$ are in $H^{\xi-1}(\X)\subset H^{\alpha+1}(\X)$. It follows that $R_2:H^{\alpha}(\X)\to H^{\alpha+1}(\X)$ is compact. Taking $R=R_1'+R_2$ therefore concludes the proof.

\section{PDE Results}

\label{sec:PDE}

We first give the proof of the existence theorem (Theorem \ref{thm: existence}) and subsequently present various results used throughout the paper with their proofs. 

\subsection{Proof of Theorem \ref{thm: existence}}

Since $f\in H^\beta(\X)$, we can take a sequence $f_n \in C^\infty(\X)$ that converges to $f$ in $H^\beta(\X)$-norm. Since $\beta>2+d/2>d/2$, we have by Sobolev embedding that $f_n$ also converges to $f$ uniformly. As $f\geq f_{\min}>0$, we thus have that for all $n$ large enough, $f_n\geq f_{\min}/2>0$. We can therefore apply Lemma \ref{lemma: boundedSolution} to infer a sequence of solutions $\{u_n\}\subset H^{2+\beta, 1+\beta/2}(\Q)$ to the PDEs \eqref{eq:PDETorus} with $f$ replaced by $f_n$. Furthermore, the solutions satisfy
\begin{align}\label{eq: estimateSequence}
	\|u_n\|_{H^{2+\beta, 1+\beta/2}} \leq C(1+\|f_n\|_{H^{\beta}}^{1+\beta/2}).
\end{align} 
for some $C>0$. Since $f_n\to f$ in $H^\beta(\X)$, this implies that $\{u_n\}$ is uniformly bounded in $H^{2+\beta, 1+\beta/2}(\Q)$. By Banach-Alaoglu, there exists a sub-sequence $u_{n_k}$ that converges weakly to a limit $u\in H^{2+\beta, 1+\beta/2}(\Q)$. By Aubins-Lions Lemma, we can assume that $u_{n_k}$ converges strongly to $u$ in $H^{\beta,\beta/2}(\Q)$.

We now show that $u$ is a weak solution to the original PDE \eqref{eq:PDETorus}. Note that since $u_n$ is a strong solution to the PDE with $f$ replaced by $f_n$, it is also a weak solution. This means that for all test functions $\varphi\in C_c^\infty(\Q)$, we have that
\[\int_\Q -u_n\del_t\varphi -\frac{\theta}{2}\nabla u_n\cdot \nabla \varphi +f_nu_n\varphi=0,\]
where $C_c^\infty(\Q)$ are the smooth functions with compact support on $\Q$.
As $n\uparrow\infty$, the first two terms converge to $\int_\Q -u\del_t\varphi - \frac{\theta}{2}\nabla u\cdot \nabla \varphi$ by weak converge of $u_n$ to $u$. For the third and last term, we observe that \[\int_\Q (f_n u_n-fu)\varphi = \int_\Q (f_n-f)u_n\varphi + \int_\Q f(u_n-u)\varphi.\] 
It is clear that $f_n$ also converges to $f$ in $L^2(\X)$ and that $\|u_n\|_2\lle \|u_n\|_{H^{2+\beta, 1+\beta/2}}$ is bounded. We thus have that the first term on the right hand side of the last display goes to 0 by Cauchy-Schwartz and the inequality $\|u_n\varphi\|_2\lle \|u_n\|_2\|\varphi\|_\infty$. The second term also goes to 0 by Cauchy-Schwartz, this time by invoking the strong convergence of $u_n$ to $u$ in $H^{\beta,\beta/2}(\Q)\supset L^2(\Q)$. It follows that
\[\int_\Q -u\del_t\varphi -\frac{\theta}{2}\nabla u\cdot \nabla \varphi +fu\varphi=0,\] and so $u\in H^{2+\beta, 1+\beta/2}(\Q)$ is a weak solution to the original PDE \eqref{eq:PDETorus}. 

Finally, since $\beta>2+d/2$ there are many results in literature from which we can deduce that $u$ is also a strong solution to the PDE \eqref{eq:PDETorus}. For instance, by Sobolev embeddings, it is clear that since $f\in H^\beta(\X)$ and $u_0\in H^{1+\beta}(\X)$, there exists $\alpha\in (0,1)$ such that $f\in C^{2+\alpha}(\X)\subset C^{2+\alpha, 1+\alpha/2}(\overline\Q)\subset C^{\alpha, \alpha/2}(\overline \Q)$, and such that $u_0\in C^{3+\alpha}(\X)\subset C^{2+\alpha}(\X)$. It follows by Corollary 5.1.22 in \cite{Lunardi1995} that there exists a unique strong solution $u_{\theta,f}\in C^{2+\alpha, 1+\alpha/2}(\Q)$ to the PDE \eqref{eq:PDETorus}. It is thus certainly the unique weak solution, so we must have $u_{\theta,f}=u\in H^{2+\beta, 1+\beta/2}(\Q)$. The estimate \eqref{eq: estimateSolution} is then obtained similarly as in the proof of Lemma \ref{lemma: boundedSolution}.

\subsection{Auxiliary Results and their Proofs}

\begin{lemma}\label{lemma:firstEstimate}
	Let $\beta>2+d/2$. There exists $C, C'>0$ such that $\forall \theta\in\Theta, F\in H^\beta(\X), u\in C(\overline\Q)$,
	\begin{align}
		&\|\eL_{\theta, F}^{-1}u\|_\infty\leq T\|u\|_\infty.\label{eq: LipsEst1bis}\\
		&\|\eL_{\theta, F}^{-1} u\|_2\leq C\|u\|_2, \label{eq: LipsEst1}\\
		&\|\eL_{\theta,F}^{-1}u\|_{H^{2,1}}\leq C'(1+\|F\|_\infty)\|u\|_2,\label{eq: lipschitz2}
	\end{align}
\end{lemma}
\begin{proof}
	Note that $v:=\eL_{\theta, F}^{-1}u$ is the solution to the PDE:
	\begin{equation*}
		\left\{\begin{aligned}
			\eL_{\theta,F}v &= u,\quad&& \text{on }\X\times (0,T),\\
			v(\cdot,0)& = 0, && \text{on }\X,
		\end{aligned}\right.
	\end{equation*}
	As we are on the torus, it has a representation in terms of the Feynman-Kac formula
	\[v(x,t) = \EE^x\left(\int_0^t u(x+\sqrt{\theta}B_s, t-s)e^{-\int_0^s \Phi\circ F(x+\sqrt{\theta}B_r)dr}ds\right),\]
	with $B$ a standard $d$-dimensional Brownian motion on the torus. Since $\Phi(F)>0$, the exponential factor in the integrand is strictly between 0 and 1. The proof of \eqref{eq: LipsEst1bis} follows since
	\[\sup_{(x,t)\in \Q} |v(x,t)|\leq \sup_{(x,t)\in\Q}\int_0^t |u(x,s)|ds\leq T\|u\|_\infty.\]
	
	For the proof of \eqref{eq: LipsEst1}, we recall the Schrödinger operator $\SSS_{\theta,F}:=\frac{\theta}{2}\Delta_x - \Phi(F)$. Consider $w_0\in L^2(\X)$, and let $w(t):= e^{t\SSS_{\theta,F}}w_0$; that is, $w$ is the solution to the PDE \[\del_t w = S_f w,\]
	with initial value $w(0)=w_0$ on the torus. Since we have \[\ip{\Delta_x w}{w}_{L^2(\X)} = -\ip{\nabla_x w}{\nabla_x w}_{L^2(\X)} = - \|\nabla_x w\|^2_{L^2(\X)},\] we derive the following inequality:
	\begin{align*}
		\ip{\SSS_{\theta,F}w}{w}_{L^2(\X)}&= \frac{\theta}{2}\ip{\Delta_x w}{w} - \ip{\Phi(F)w}{w}\\
		&=-\frac{\theta}{2}\|\nabla_x w\|_{L^2(\X)}^2 - \int_\X \Phi(F)(x)w(x,t)^2dx\\
		&\leq - \int_\X \Phi(F)(x)w(x,t)^2dx \text{ (since }\theta>0)\\
		&\leq -f_{\min} \|w(t)\|_{L^2(\X)}^2.
	\end{align*}
	 Let now 
	\[y(t):=\|w(t)\|_{L^2(\X)}^2 = \int_\X |w(x,t)|^2 dx.\]
	This implies in particular that $y'(t)=2\ip{\del_t w}{w}_{L^2(\X)}$. It follows that
	\begin{align*}
		y'(t)= 2\ip{\del_t w}{w} = 2\ip{\SSS_{\theta,F}w}{w}\leq -2 f_{\min} y(t).
	\end{align*}
	By Gronwall's inequality, the above display gives
	\begin{align}
		&y(t) \leq e^{-t2f_{\min}}y(0)\nonumber\\
		\iff &\|w(t)\|_{L^2(\X)} \leq e^{-tf_{\min}}\|w(0)\|_{L^2(\X)}\nonumber\\
		\implies &\|e^{t\SSS_{\theta,F}}\|_{L^2(\X)\to L^2(\X)}\leq e^{-tf_{\min}},\label{eq: opnormBound}
	\end{align}
	where for an operator $A:L^2(\X)\to L^2(\X), \ \|A\|_{L^2(\X)\to L^2(\X)}:=\sup_{\|h\|_{L^2(\X)}=1} \|Ah\|_{L^2(\X)}$.
	Let now $v:=\eL^{-1}_{\theta, F}u$ as before. Since $u\in C(\overline\Q)$, Theorem 5.1.11 in \cite{Lunardi1995} implies that $v$ has the following representation:
	\[v(x,t) = \int_0^t e^{(t-s)\SSS_{\theta,F}} u(\cdot, s)ds(x).\]
	By \eqref{eq: opnormBound}, it follows that 
	\begin{align*}
		\|v(\cdot, t)\|_{L^2(\X)} &\leq \int_0^t \|e^{(t-s)\SSS_{\theta,F}}\|_{L^2(\X)\to L^2(\X)} \|u(\cdot, s)\|_{L^2(\X)}ds\\
		&\leq \int_0^t e^{-f_{\min} (t-s)}\|u(\cdot, s)\|_{L^2(\X)}ds.
	\end{align*}
	Using Cauchy-Schwartz, we further obtain that
	\begin{align*}
		\|v\|_{L^2(\Q)}^2&= \int_0^T\|v(\cdot, t)\|_{L^2(\X)}^2dt\\
		&= \int_0^T \Bigl(\int_0^t e^{-2f_{\min}(t-s)}ds\Bigr)\Bigl(\int_0^t \|u(\cdot, s)\|_{L^2(\X)}^2ds\Bigr)dt\\
		&\leq \frac{1}{2f_{\min}}\int_0^T\int_0^t \|u(\cdot, s)\|_{L^2(\X)}^2dsdt\\
		&= \frac{1}{2f_{\min}} \int_0^t(T-s)\|u(\cdot,s)\|_{L^2(\X)}^2ds \text{ (by Fubini)}\\
		&\leq \frac{T}{2f_{\min}}\int_0^T\|u(\cdot, s)\|_{L^2(\X)}^2ds\\
		&\leq \frac{T}{2f_{\min}}\|u\|_{L^2(\Q)}^2.
	\end{align*}
	The proof of \eqref{eq: LipsEst1} is concluded by taking $C=\sqrt{T/(2f_{\min})}$.
	
	Finally, we know (Remark 15.1 in Chapter 4 of \cite{lionsVol2}) that $L_\theta:=\del_t-\dfrac{\theta}{2}\Delta_x$ is an isomorphism of $H^{2,1}_{B,0}(\Q)$ onto $L^2(\Q)$. It follows that there exists $C_\theta > 0$ such that
	\begin{align*}
		\|\eL_{\theta,F}^{-1}u\|_{H^{2,1}}&\leq C_\theta \|L_\theta (\eL_{\theta,F}^{-1}u)\|_2\\
		&= C_\theta \|-u + \Phi(F)\eL_{\theta,F}^{-1}u\|_2\\
		&\leq C_\theta(1+C\|\Phi(F)\|_\infty)\|u\|_2 \text{ (by \eqref{eq: LipsEst1})}\\
		&\lle C_\theta(1+\|F\|_\infty)\|u\|_2,
	\end{align*}
	where the last inequality follows from (6.1) of Lemma 29 in \cite{nicklVdG}. The proof is concluded if we can show that $\theta\mapsto C_\theta$ is a bounded map, which by assumed compactness of $\Theta$ certainly follows if it is continuous. Now $\|L_\theta v - L_\tau v\|_{2} = \|\frac{\theta-\tau}{2}\Delta_x v\|_2
	\lesssim |\theta-\tau| \|v\|_{H^{2, 1}}$, for any $v\in H_{B,0}^{2, 1}(\Q)$, and the result follows. 
\end{proof}

\begin{lemma}\label{lemma: bounds}
	Let $\beta>2+d/2$. There exists a constant $D$ such that for all $ \theta, \theta_0\in\Theta$ and $F, F_0, h\in H^\beta(\X)$, 
	\begin{enumerate}[label=(\roman*),leftmargin=*, itemsep=0.5ex, before={\everymath{\displaystyle}}]
		\item $ \|K_\theta F - K_\theta F_0\|_2 \leq D \|F-F_0\|_2,$
		\item $\|(I_{\theta,F} - I_{\theta,F_0})h\|_2\leq D\|F-F_0\|_2\|h\|_\infty$,
		\item $\|K_\theta F -K_{\theta_0}F\|_{H^{2,1}(\Q)}\leq D|\theta-\theta_0|(1+\|F\|_2)(1+\|F\|_\infty)$,
		\item $\|\dot K_\theta F - \dot K_\theta F_0\|_{H^{2,1}(\Q)}\leq D\|F-F_0\|_2(1+\|F\|_\infty),$
		\item $\|K_\theta F - K_{\theta_0}F-(\theta-\theta_0)\dot K_{\theta_0}F\|_{H^{2,1}(\Q)}\leq D|\theta-\theta_0|^2(1+\|F\|_2)(1+\|F\|_\infty)^2$,
		\item $\|K_{\theta}(F+h) - K_{\theta}F - I_{\theta,F}h\|_{H^{2,1}(\Q)}\leq D\|h\|_2\|h\|_\infty(1+\|F+h\|_\infty)$,
		\item $\|I_{\theta,F} h\|_{H^{2,1}}\le D \|h\|_2\|K_\theta F\|_\infty$.
	\end{enumerate}
\end{lemma}

\begin{proof}

We begin by noticing that all functions in the left hand sides of the statements of the lemma vanish on the parabolic boundary of $\Q$. Thus they are all in the range of the isomorphisms $\eL_{\tau, G}^{-1}:   L^2(Q)\to H_{B,0}^{2,1}(\Q)$,
for the different possible values of $(\tau, G)\in\Theta\times H$. 

(i). Observe that,
\begin{align*}
	\eL_{\theta,F}(K_\theta F - K_{\theta}F_0) &= 0 - \eL_{\theta,F}K_{\theta}F_0\\
&= (\eL_{\theta,F_0} - \eL_{\theta,F})K_\theta F_0
	= (\Phi(F)-\Phi(F_0))K_{\theta}F_0.
\end{align*}
It follows by \eqref{eq: LipsEst1} that,
\begin{align*}
	\|K_\theta F- K_\theta F_0\|_2 & = \left\|\eL_{\theta,F}^{-1}\left[(\Phi(F_0)-\Phi(F))K_\theta F_0\right]\right\|_2\\
	&\leq C\|(\Phi(F)-\Phi(F_0))K_\theta F_0\|_2
	\leq C\|\Phi(F)-\Phi(F_0)\|_2\|K_{\theta}F_0\|_\infty.
\end{align*}
We conclude by applying \eqref{eq: PhiCondition} and \eqref{eq: uniformBound}. 

(ii). The definition of $I_{\theta,F}$ implies that
\begin{align*}
	\eL_{\theta,F}\left(I_{\theta, F}h - I_{\theta,F_0}h\right)& =h (\Phi'(F)K_\theta F - \Phi'(F_0)K_\theta F_0) + (\eL_{\theta,F_0} - \eL_{\theta,F})(I_{\theta,F_0}h)\\
	&=h(\Phi'(F) - \Phi'(F_0))K_\theta F + h\Phi'(F_0) (K_\theta F - K_\theta F_0)\\
	&\qquad+ (\Phi(F)-\Phi(F_0))I_{\theta, F_0}h.
\end{align*}
Using \eqref{eq: LipsEst1}, we thus obtain that,
\begin{align*}
	\|(I_{\theta,F} - I_{\theta,F_0})h\|_2 &\lle \|h\|_\infty \|K_\theta F\|_\infty\|\Phi'(F)-\Phi'(F_0)\|_2\\
	&\qquad+ \|h\|_\infty \|\Phi'(F_0)\|_\infty \|K_\theta F-K_\theta F_0\|_2\\
	&\qquad+ \|I_{\theta,F_0}h\|_\infty \|\Phi(F)-\Phi(F_0)\|_2.
\end{align*}
By \eqref{eq: uniformBound} and the fact that $\Phi'$ is Lipschitz, the first term is of order $\|h\|_\infty\|F-F_0\|_2$. The second term is of the same order by item (i) of the lemma. Besides by \eqref{eq: LipsEst1bis}, $\|I_{\theta,F_0}h\|_\infty \lle \|h\|_\infty$ and since $\Phi$ is Lipschitz, the third term is therefore also of order $\|h\|_\infty\|F-F_0\|_2$.

(iii). Application of $\eL_{\theta,F}$ to the left hand side yields,
\begin{align*}
	\eL_{\theta,F}(K_\theta F - K_{\theta_0}F)&= (\eL_{\theta_0,F}-\eL_{\theta,F})K_{\theta_0}F
	= \frac{\theta_0-\theta}{2}\Delta_x K_{\theta_0}F.
\end{align*}
Inequality \eqref{eq: lipschitz2} and the fact that $\|\Delta_xu\|_2\lle\|u\|_{H^{2,1}}$ for $u\in H^{2,1}(\Q)$ imply
\begin{align*}
	\|(K_{\theta}-K_{\theta_0})F\|_{H^{2,1}}&\lle |\theta-\theta_0|(1+\|F\|_\infty)\|\Delta_x K_{\theta_0}F\|_2\\
	&\leq |\theta-\theta_0| (1+\|F\|_\infty)\left(\|\Delta_x(K_{\theta_0}F - K_{\theta_0}F_0)\|_2 + \|\Delta_x K_{\theta_0}F_0\|_2\right)\\
	&\lle |\theta-\theta_0| (1+\|F\|_\infty)\left(\|K_{\theta_0}F - K_{\theta_0}F_0\|_{H^{2,1}} + 1 \right).
\end{align*}
The proof of this item is concluded by applying \eqref{eq: lipschitz2} and observing that, 
\begin{align}\label{eq: lipschitz_intermediate}
	\|K_\theta F- K_\theta F_0\|_{H^{2,1}}&=\|\eL_{\theta,F_0}^{-1}((\Phi(F)-\Phi(F_0))K_\theta F)\|_{H^{2,1}}\\
	&\leq C'(1+\|F_0\|_\infty)\|\Phi(F)-\Phi(F_0)\|_2\|K_\theta F\|_\infty \nonumber\\
	&\lle \|F-F_0\|_2\|K_\theta F\|_\infty 
	\leq U\|F-F_0\|_2 \lle 1 + \|F\|_2.\nonumber
\end{align}

(iv). The definition of $\dot K_\theta F$ implies that,
\begin{align*}
	\eL_{\theta,F}(\dot K_{\theta}F - \dot K_{\theta}F_0) &= -\frac{\theta}{2}(\Delta_x K_\theta F - \Delta_x K_\theta F_0) + (\eL_{\theta,F_0} - \eL_{\theta, F})(\dot K_\theta F_0)\\
	&=\frac{\theta}{2}\Delta_x(K_\theta F_0- K_\theta F) + (\Phi(F) - \Phi(F_0))\dot K_\theta F_0.
\end{align*}
By \eqref{eq: lipschitz2} we have,
\begin{align*}
	\|\dot K_\theta F - \dot K_\theta F_0\|_{H^{2,1}(\Q)}&\lesssim(1+\|F\|_\infty)\left(\theta \|\Delta_x(K_\theta F- K_\theta F_0)\|_2 +\|(\Phi(F) - \Phi(F_0))\dot K_\theta F_0\|_2\right)\\
	&\lle (1+\|F\|_\infty)\left(\|K_\theta F - K_\theta F_0\|_{H^{2,0}} + \|\Phi(F)-\Phi(F_0)\|_2\|\dot K_\theta F_0\|_\infty\right)\\
%	&\lle (1+\|F\|_\infty) \left(\|K_\theta F - K_\theta F_0\|_{H^{2,1}} +\|F-F_0\|_2\right)\\
	&\lle (1+\|F\|_\infty) \|F-F_0\|_2
\end{align*}
where in the last inequality we used \eqref{eq: lipschitz_intermediate} and
$\|\dot K_\theta F_0\|_\infty\lle  \|K_\theta F_0\|_{H^{2+\beta,1+\beta/2}}$. This last inequality is a consequence of \eqref{eq: LipsEst1bis} and the fact that since $\beta>2+d/2$, we have the embedding $H^{\beta,\beta/2}(\Q)\subset C^{1,1/2}(\Q)\subset L^\infty(\Q)$. Namely, we have
\begin{align*}
	\|\dot K_\theta F_0\|_\infty \leq \frac{T}{2}\|\Delta_x K_\theta F\|_\infty \lle \|\Delta_x K_\theta F_0\|_{H^{\beta,\beta/2}}\lle \|K_\theta F_0\|_{H^{2+\beta,1+\beta/2}}.
\end{align*}
Of course, to complete the proof we need to show that $\|K_\theta F\|_{H^{2+\beta,1+\beta/2}}$ is bounded above by a constant independent of $\theta$. This can be done by bounding it above by 
\begin{align*}
	&\|K_{\theta_0}F_0\|_{H^{2+\beta, 1+\beta/2}} + \|K_\theta F_0 - K_{\theta_0}F_0\|_{H^{2+\beta,1+\beta/2}}\\
	&= \|K_{\theta_0}F_0\|_{H^{2+\beta, 1+\beta/2}} + \frac{|\theta-\theta_0|}{2}\|\eL_{\theta,F_0}^{-1}[\Delta_x K_{\theta_0}F_0]\|_{H^{2+\beta,1+\beta/2}}\\
	&\lle 1 + \|\Delta_x K_{\theta_0}F_0\|_{H^{\beta,\beta/2}} \text{ (by \eqref{eq:refinedH} of Lemma \ref{lemma: refinedLipschitz} with $\eta=\beta$)}\\
	&\lle 1.
\end{align*}

(v). Observe that,
\begin{align*}
	&\eL_{\theta,F}\left(K_\theta F - K_{\theta_0} F -(\theta-\theta_0)\dot K_{\theta_0}F\right)\\
	&\quad= (\eL_{\theta_0,F}-\eL_{\theta,F})(K_{\theta_0} F) - (\theta -\theta_0)\eL_{\theta_0,F}(\dot K_{\theta_0}F)+(\theta-\theta_0)(\eL_{\theta_0,F}-\eL_{\theta,F})(\dot K_{\theta_0}F)\\
	&\quad= \frac{\theta_0-\theta}{2}\Delta_x K_{\theta_0}F - \frac{\theta_0-\theta}{2}\Delta_x K_{\theta_0}F - \frac{(\theta-\theta_0)^2}{2}\Delta_x \dot K_{\theta_0}F\\
	&\quad=  -\frac{(\theta-\theta_0)^2}{2}\Delta_x \dot K_{\theta_0}F.
\end{align*}
It follows by two applications of \eqref{eq: lipschitz2} that,
\begin{align*}
	\|(K_\theta -K_{\theta_0}-(\theta-\theta_0)\dot K_{\theta_0})F\|_{H^{2,1}}&\lle |\theta-\theta_0|^2(1+\|F\|_\infty)\|\Delta_x \dot K_{\theta_0}F\|_2\\
	&\lle|\theta-\theta_0|^2(1+\|F\|_\infty)\|\dot K_{\theta_0}F\|_{H^{2,1}}\\
	& = |\theta-\theta_0|^2 (1+\|F\|_\infty)\|\eL_{\theta_0,F}^{-1}(\Delta_x K_{\theta_0}F / 2)\|_{H^{2,1}}\\
	&\lle |\theta-\theta_0|^2 (1+\|F\|_\infty)^2(1+\|F\|_2),
	\end{align*}
where we also used $\|\Delta_xK_{\theta_0}F\|_2 \lle 1+\|F\|_2$, which was shown in the proof of item (iii).

(vi). The proof of this item is similar to the proof of the second part of Lemma~\ref{lemma: devs}. The only difference is that in the first step when bounding the $H^{2,1}$-norm of $K_{\theta}(F+h) - K_{\theta}F - I_{\theta,F}h$ instead of its $L^2$-norm, \eqref{eq: lipschitz2} is used instead of \eqref{eq: LipsEst1}, hence the added factor $(1+\|F\|_\infty)$.

(vii). By definition $I_{\theta, F}h=\eL_{\theta,F}^{-1}(h\Phi'(F)K_\theta F)$. Thus the result follows from
\eqref{eq: lipschitz2}.

\end{proof}

\begin{lemma}\label{lemma: refinedBounds}
	Let $\beta>2+d/2$. For all $\eta\in (0,1)$ and $R>0$, there exists a positive constant $D(\eta,R)$ such that for all $\theta\in\Theta$, $F\in H^\beta(\X)$ with
	$\|F\|_{H^\beta(\X)}\le R$ and  $h\in H^\beta(\X)$,
	\begin{enumerate}[label=(\roman*),leftmargin=*, itemsep=0.5ex, before={\everymath{\displaystyle}}]
		\item $\|\dot K_\theta F - \dot K_\theta F_0\|_{H^{2+\eta,1+\eta/2}(\Q)}\leq D(\eta,R)\|F-F_0\|_{H^\beta(\X)}$,
		\item $\|K_\theta F - K_{\theta_0}F-(\theta-\theta_0)\dot K_{\theta_0}F\|_{H^{2+\eta,1+\eta/2}(\Q)}\leq D(\eta,R)|\theta-\theta_0|^2$,
		\item $\|K_{\theta_0}(F+h) - K_{\theta_0}F - I_{\theta_0,F}h\|_{H^{2+\eta,1+\eta/2}(\Q)}\leq D(\eta,R)\|h\|_{H^\beta}^2(1+\|h\|_{H^\beta}+\|h\|_{H^\beta}^2),$
		\item $\|I_{\theta, F}h - I_{\theta,F_0}h\|_{H^{2+\eta,1+\eta/2}(\Q)}\leq D(\eta, R)\|F-F_0\|_{H^\beta(\X)}\|h\|_{H^{\beta}(\X)}.$
	\end{enumerate}
\end{lemma}

\begin{proof}
	The four items in the lemma are specialisations of the following more general bounds. For $F\in H^\beta(\X)$,
	define
	$$C(F):=1+\|\Phi(F)\|_{C^\eta(\X)} + \|\Phi(F)\|_{C^\eta(\X)}\|F\|_\infty.$$
	Then for $\eta\in(0,1)$, there exists a postive constant $E_\eta$ such that for all $\theta\in\Theta$ and $F,h \in H$:
	\begin{enumerate}[label=(\alph*), itemsep = 10pt]
		\item $\|\dot K_\theta F -\dot K_\theta F_0\|_{H^{2+\eta, 1+\eta/2}(\Q)}
		\leq E_{\eta}C(F)^2\|\Phi(F)-\Phi(F_0)\|_{C^\eta(\X)}$,
		\item $\|K_\theta F-K_{\theta_0}F-(\theta-\theta_0)\dot K_{\theta_0}F\|_{H^{2+\eta,1+\eta/2}(\Q)}
		\leq E_{\eta}C(F)^4|\theta-\theta_0|^2$,
		\item $\|K_{\theta_0}(F+h)- K_{\theta_0}F - I_{\theta_0,F}h\|_{H^{2+\eta,1+\eta/2}(\Q)} \leq E_{\eta}C(F+h)\times \\ \\
		{}\qquad\qquad\qquad\times \Bigl[\|\Phi(F+h)-\Phi(F)\|_{C^\eta(\X)}\|I_{\theta_0,F}h\|_{H^{\eta,\eta/2}(\Q)} \\ \\
		{}\qquad\qquad\qquad\qquad     +\|\Phi(F+h)-\Phi(F)-h\Phi'(F)\|_{C^\eta(\X)}\|K_{\theta_0}F\|_{H^{\eta,\eta/2}(\Q)}.\Bigr]$.
		\item $\|I_{\theta_0, F}h - I_{\theta_0,F_0}h\|_{H^{2+\eta,1+\eta/2}(\Q)} \leq E_\eta C(F)\times\\ \\ {}\qquad \qquad \times \Bigl[ \|h\|_{C^\eta(\X)}\|\Phi'(F)\|_{C^\eta(\X)}\|K_\theta F - K_\theta F_0\|_{H^{\eta,\eta/2}} \\ \\ {}\qquad \qquad \qquad + \|h\|_{C^\eta(\X)}\|\Phi'(F)-\Phi'(F_0)\|_{C^\eta(\X)}\|K_{\theta_0} F_0\|_{H^{\eta,\eta/2}(\Q)}\\ \\ {}\qquad \qquad \qquad + \|\Phi(F)-\Phi(F_0)\|_{C^\eta(\X)}\|I_{\theta_0,F_0}h\|_{H^{\eta,\eta/2}(\Q)}\Bigr]$
	\end{enumerate}
	
	These bounds are themselves instances of application of the Lipschitz estimate in Lemma~\ref{lemma: refinedLipschitz} below.
	We proceed in three steps. First we derive  (i)--(iv) of Lemma~\ref{lemma: refinedBounds} from (a)--(b). Next we deduce (a)--(d) from 
	Lemma~\ref{lemma: refinedLipschitz}.
	
	Assume then that (a)--(d) are true and let $F\in H\cap B_{H^\beta}(R)$. We first note that $\|F\|_\infty+\|\nabla F\|_\infty\lle 
	\|F\|_{H^\beta(\X)} < R$, by Sobolev embedding, since $\beta-1>d/2$.  
	Furthermore since $\Phi(F)+\|\nabla \Phi(F)\|\lesssim 1+|F|+\|\nabla F\|$ under the assumptions on $\Phi$, 
	we have $\|\Phi(F)\|_{C^1(\X)}\lesssim 1+\|F\|_{C^1(\X)}$.
	Since $\eta < 1 $, it follows that $\|\Phi(F)\|_{C^\eta}\lle 1+\|F\|_{C^1(\X)}\lle 1+\|F\|_{H^\beta(\X)}$, and hence
	the constant $C(F)$ in (a)--(d) is bounded by $1+R^2$ for $F\in H\cap B_{H^\beta}(R)$.
	
	Item (ii) of Lemma~\ref{lemma: refinedBounds} then immediately follows from (b).
	
	Furthermore, keeping in mind that $\Phi, \Phi'$ and $\Phi''$ are all Lipschitz with bounded derivatives, we similarly have
	\begin{align*}
		\bigl|\Phi(F+G)-\Phi(F)\bigr|&\lesssim |G|,\\
		\bigl\|\nabla\Phi(F+G)-\nabla \Phi(F)\bigr\|&\lesssim\|\nabla G\|+|G|\,\|\nabla F\|,\\
		\bigl| \Phi'(F+G)-\Phi'(F)\bigr|&\lle |G|,\\
		\bigl\|\nabla\Phi'(F+G) - \nabla \Phi'(G)\|&\lle \|\nabla G\| + |G|\|\nabla F\|,\\
		\bigl\|\Phi(F+G)-\Phi(F)-G\Phi'(F)\bigr\|&\lesssim |G|^2,\\
		\bigl\|\nabla\Phi(F+G)-\nabla \Phi(F)-\nabla(G\Phi'(F))\bigr\|&\lesssim\|\nabla F\|\,|G|^2+|G|\,\|\nabla G\|.
	\end{align*}
	The first two inequalities readily give that 
	$\|\Phi(F+G)-\Phi(F)\|_{C^1(\X)}\lesssim (1+\|\nabla F\|_\infty)\|G\|_{C^1(\X)}$, which implies that
	$\|\Phi(F+G)-\Phi(F)\|_{C^\eta(\X)}\lesssim (1+R) \|G\|_{H^\beta(\X)}$, for 
	$\|\nabla F\|_\infty\lesssim \|\nabla F\|_{H^{\beta-1}(\X)}\le \|F\|_{H^\beta(\X)}\le R$, as $\beta-1>d/2$ by assumption. The third and fourth inequality work the same way to also imply $\|\Phi'(F+G) - \Phi'(F)\|_{C^\eta(\X)}\lle (1+R)\|G\|_{H^\beta(\X)}$.   
	The fifth and sixth inequalities similarly yield that
	$\|\Phi(F+G)-\Phi(F)-G\Phi'(F)\|_{C^\eta(\X)}\lesssim R \|G\|_{H^\beta(\X)}^2$, since
	$\|G^2\|_\infty\le \|G\|_{H^\beta(\X)}^2$ and $\||G|\nabla G\|_\infty\le
	\|G\|_\infty\|\nabla G\|_{\infty}\le \|G\|_{H^\beta(\X)}^2$, for $\beta-1>d/2$. 
	
	%observe that for any function $G\in H^{\beta}(\X)$,
	%\begin{align*}
	%	\|\Phi(F+G)-\Phi(F)\|_{H^\eta(\X)}&\lle \|\Phi(F+G)-\Phi(F)\|_{H^1(\X)}\\
	%	&= \|\Phi(F+G) - \Phi(F)\|_2 + \|\nabla (\Phi(F+G)-\Phi(F))\|_2\\
	%	&\lle \|G\|_2 + \|\Phi'(F+G)\nabla(F+G) - \Phi'(F)\nabla F\|_2\\
	%	&\leq \|G\|_2 + \|\Phi'(F+G)\nabla G\|_2 + \|(\Phi'(F+G)-\Phi'(F))\nabla F\|_2\\
	%	&\lle \|G\|_2 + \|\Phi'(F+G)\|_\infty \|\nabla G\|_2 + \|\Phi'(F+G)-\Phi'(F)\|_2\|\nabla F\|_\infty\\
	%	&\lle (1 + \|\nabla F\|_\infty)\|G\|_2 + (1+\|F+G\|_\infty)\|\nabla G\|_2.
	%\end{align*}
	%Since $\beta > 2 + d > 1 +d/2,$ we have in particular that $\beta-1>d/2$. This yields, \[\|\nabla F\|_\infty \lle \|\nabla F\|_{H^{\beta-1}(\X)}\lle \|F\|_{H^{\beta}(\X)}<M.\]
	%The previous two displays combined imply the following bound, where the inferred constant depends on $M$.
	%\[\|\Phi(F+G)-\Phi(F)\|_{H^\eta(\X)} \lle \|G\|_2 + \|\nabla G\|_2 = \|G\|_{H^1(\X)}.\]
	
	Item (i) of Lemma~\ref{lemma: refinedBounds} then follows from applying this bound with $G=F_0 - F$.
	
	Item (iii) follows from applying this with $G=h$, in combination with the two 
	following bounds, where the multiplicative constants are allowed to depend on $R$. 
	\begin{align*}
		\|I_{\theta_0,F}h\|_{H^{\eta,\eta/2}(\Q)} &= \|\eL_{\theta_0,F}^{-1}(h\Phi'(F)K_{\theta_0}F)\|_{H^{\eta,\eta/2}(\Q)}
		\lle \|\eL_{\theta_0, F}^{-1}(h\Phi'(F)K_{\theta_0}F)\|_{H^{2,1}(\Q)}\\
		&\leq C'(1+\|F\|_\infty)\|h\Phi'(F)K_{\theta_0}F\|_2 \qquad\text{ (by \eqref{eq: lipschitz2})}\\
		&\lle \|h\|_\infty\|K_{\theta_0} F\|_\infty\|\Phi'(F)\|_2\lle \|h\|_{H^{\beta}(\X)}.
	\end{align*} 
	\begin{align*}
		\|K_{\theta_0}F\|_{H^{\eta,\eta/2}(\Q)} &\leq \|K_{\theta_0}F - K_{\theta_0}F_0\|_{H^{\eta,\eta/2}(\Q)} + \|K_{\theta_0}F_0\|_{H^{\eta,\eta/2}(\Q)}\\
		&\lle 1+\|F\|_2 + \|K_{\theta_0}F_0\|_{H^{2+\beta,1+\beta/2}(\Q)} \qquad \text{ (by \eqref{eq: lipschitz_intermediate})}.
	\end{align*}
	The right side is bounded by a multiple of $R+1$, since $K_{\theta_0}F_0\in H^{2+\beta,1+\beta/2}(\Q)$ by assumption. We also use Lemma 29 in \cite{nicklVdG} to bound $C(F+h)$ by a constant multiple of
	 \begin{align*}
	 	&1+\|F+h\|_{C^{\eta}} +\|F+h\|_{\infty}(1+\|F+h\|_{\infty})\\
	 	&\lle 1+\|F\|_{H^\beta}+\|F\|_{H^\beta}^2+\|h\|_{H^\beta}(1+\|F\|_{H^\beta})+\|h\|_{H^\beta}^2\\
	 	&\leq 1+R+R^2+(1+R)\|h\|_{H^\beta}+\|h\|^2_{H^\beta},
	 \end{align*}
	and this explains the last factor on the right hand side of item (iii).
	
	Finally, item (iv) follows from applying the $\Phi$ related bounds with $G=F_0-F$ combined with the fact that for $\beta>1+d/2>\eta+d/2$, $\|h\|_{C^\eta(\X)}\lle\|h\|_{H^\beta(\X)}$, and with the fact that
	\begin{align*}
		\|K_{\theta_0}F-K_{\theta_0}F_0\|_{H^{\eta,\eta/2}(\Q)}\lle \|K_{\theta_0}F-K_{\theta_0}F_0\|_{H^{2,1}(\Q)} \lle \|F-F_0\|_2\lle \|F-F_0\|_{H^\beta(\X)},
	\end{align*}
	where we used \eqref{eq: lipschitz_intermediate} in the before last inequality.
	
	%\begin{align*}
	%	&\|\Phi(F+h)-\Phi(F) - h\Phi'(F)\|_{H^\eta(\X)}\\
	%	&\lle \|\Phi(F+h)-\Phi(F) - h\Phi'(F)\|_{H^1(\X)}\\
	%	&= \|\Phi(F+h)-\Phi(F) - h\Phi'(F)\|_2 + \|\nabla\left(\Phi(F+h)-\Phi(F) - h\Phi'(F)\right)\|_2\\
	%	&\leq\|h^2\|_2 + \|\Phi'(F+h)\nabla(F+h)-\Phi'(F)\nabla F - h\Phi''(F)\nabla F - \Phi'(F)\nabla h\|_2\\
	%	&\lle \|h\|_\infty\|h\|_2 + \|\left[\Phi'(F+h)-\Phi'(F)-h\Phi''(F)\right]\nabla F\|_2 + \|\left[\Phi'(F+h)-\Phi'(F)\right]\nabla h\|_2\\
	%	&\lle \|\nabla F\|_\infty \|h\|_\infty\|h\|_2 + \|h\|_\infty\|\nabla h\|_2\\
	%	&\lle \|h\|_\infty\|h\|_{H^1(\X)}\\
	%	&\lle \|h\|^2_{H^{\beta(\X)}}.
	%\end{align*}

Next we derive items (a)--(d) from Lemma~\ref{lemma: refinedLipschitz}. Note that we will exclusively be using \eqref{eq:refinedC}.

(a). The proof of item (iv) in Lemma~\ref{lemma: bounds} in combination with Lemma~\ref{lemma: refinedLipschitz} 
and \eqref{EqSobolevProductOne} yield the following upper bound (up to constants independent of $R$) on $\|\dot K_\theta F - \dot K_\theta F_0\|_{H^{2+\eta,1+\eta/2}(\Q)}$,
\begin{align*}
	&C(F) \left(\|\Delta_x(K_\theta F - K_\theta F_0)\|_{H^{\eta,\eta/2}} + \|\Phi(F)-\Phi(F_0)\|_{C^\eta}\|\dot K_\theta F_0\|_{H^{\eta,\eta/2}} \right).
\end{align*}
By  item (iii) of Lemma~\ref{lemma: bounds}, the map $\theta\mapsto K_\theta F_0$ is continuous with respect to 
the $H^{2,1}(\Q)$-norm and hence the map $\theta\mapsto \| K_\theta F_0\|_{H^{2,1}}$ is uniformly bounded over the compact set $\Theta$. Using \eqref{eq: lipschitz2}, we can thus bound $\|\dot K_\theta F_0\|_{H^{\eta,\eta/2}}$ independently of $R$ by 
\begin{align*}
	\|\dot K_\theta F_0\|_{H^{2,1}} =  \|\eL_{\theta,F_0}^{-1}(\Delta_x K_\theta F_0/2)\|_{H^{2,1}} \lle \|\Delta_x K_\theta F_0\|_2\lle \|K_\theta F_0\|_{H^{2,1}}<\infty.
\end{align*}
Since $\|\Delta_xu\|_{H^{\eta,\eta/2}}\lle \|u\|_{H^{2+\eta,1+\eta/2}}$, we can further upper bound 
the before last display (up to constants independent of $R$) by,
\begin{align*}
	&C(F)\left(\|K_\theta F - K_\theta F_0\|_{H^{2+\eta,1+\eta/2}} + \|\Phi(F)-\Phi(F_0)\|_{C^\eta}\right).
\end{align*}
Since $K_\theta F - K_\theta F_0= \eL_{\theta, F}^{-1}\bigl(\Phi(F)-\Phi(F_0)\bigr)K_\theta F_0\bigr)$,
another application of Lemma~\ref{lemma: refinedLipschitz} yields that 
\begin{align*}
	\|K_\theta F - K_\theta F_0\|_{H^{2+\eta,1+\eta/2}} 
	&\leq C(F)\|(\Phi(F)-\Phi(F_0))K_\theta F_0\|_{H^{\eta,\eta/2}}\\
	&\lle C(F)\|\Phi(F)-\Phi(F_0)\|_{C^\eta}\|K_\theta F_0\|_{H^{\eta,\eta/2}},
\end{align*}
by \eqref{EqSobolevProductTwo}. The proof of (a) is complete upon using that $\theta\mapsto \|K_\theta F_0\|_{H^{\eta,\eta/2}(\Q)}$ is bounded on $\Theta$, since the $H^{\eta,\eta/2}(\Q)$-norm is weaker than the $H^{2,1}(\Q)$-norm.

(b). Using the method of proof of item (v) in Lemma~\ref{lemma: bounds}, we have by Lemma~\ref{lemma: refinedLipschitz} that
\begin{align*}
	&\|K_\theta F-K_{\theta_0}F-(\theta-\theta_0)\dot K_{\theta_0}F\|_{H^{2+\eta,1+\eta/2}}\\
	&\qquad\leq C(F)\,|\theta-\theta_0|^2\|\Delta_x \dot K_{\theta_0}F\|_{H^{\eta,\eta/2}}\\
	&\qquad\lle C(F) |\theta-\theta_0|^2\left(\|\dot K_{\theta_0}F - \dot K_{\theta_0}F_0\|_{H^{2+\eta,1+\eta/2}} + \|\dot K_{\theta_0}F_0\|_{H^{2+\eta,1+\eta/2}}\right)\\
	&\qquad\lle C(F)|\theta-\theta_0|^2 \left(C(F)^2\|\Phi(F)-\Phi(F_0)\|_{C^\eta} + 1\right)\\
	&\qquad\lle C(F)^3|\theta-\theta_0|^2(1+\|\Phi(F)\|_{C^\eta}),
\end{align*}
where item (a) was used in the second last inequality.

(c). The proof of this item starts the same as the proof of the second part of Lemma~\ref{lemma: devs}, and next Lemma~\ref{lemma: refinedLipschitz} is applied instead of \eqref{eq: LipsEst1}.

(d). For this item we start exactly as in the proof of item (ii) in Lemma \ref{lemma: bounds} but apply Lemma \ref{lemma: refinedLipschitz} instead of \eqref{eq: LipsEst1}. We also use \eqref{EqSobolevProductTwo} and the fact that for two functions $a,b\in C^\eta(\X)$, we have $\|ab\|_{C^\eta}\lle \|a\|_{C^\eta}\|b\|_{C^\eta}$. 

\end{proof}

\begin{lemma}\label{lemma: refinedLipschitz}
	For $\eta\in(0,1)$, there exists a constant $C_1>0$ such that $\forall \theta\in \Theta, F\in C^\eta(\X)$ and  $u\in H^{\eta,\eta/2}(\Q)$ continuous
	\begin{align}\label{eq:refinedC}
		\|\eL_{\theta, F}^{-1}(u)\|_{H^{2+\eta, 1+\eta/2}(\Q)}\leq C_1 \bigl(1+\|\Phi(F)\|_{C^\eta(\X)} + \|\Phi(F)\|_{C^\eta(\X)}\|F\|_\infty\bigr)\|u\|_{H^{\eta,\eta/2}}.
	\end{align}
	More generally, for any $\eta>0$, let $M:=\lceil \eta/2 \rceil$. There exists a constant $C_2>0$ such that $\forall \theta\in \Theta, F\in C^\eta(\X)$ and $u\in H^{\eta,\eta/2}(\Q)$ continuous
	\begin{align}\label{eq:refinedCIterated}
		\|\eL_{\theta, F}^{-1}(u)\|_{H^{2+\eta, 1+\eta/2}(\Q)}\leq C_2\left(\sum_{k=0}^M\|\Phi(F)\|^k_{C^\eta} + \|\Phi(F)\|_{C^\eta}^M\|F\|_\infty\right)\|u\|_{H^{\eta,\eta/2}}.
	\end{align}
	Furthermore, for $d\leq 3$, if $\eta>d/2$, then there exists a constant $C_3>0$ such that $\forall \theta\in \Theta, F\in H^\eta(\X)$ and $u\in H^{\eta,\eta/2}(\Q)$ continuous
	\begin{align}\label{eq:refinedH}
		\|\eL_{\theta, F}^{-1}(u)\|_{H^{2+\eta, 1+\eta/2}(\Q)}\leq C_3\left(\sum_{k=0}^M\|\Phi(F)\|^k_{H^\eta} + \|\Phi(F)\|_{H^\eta}^M\|F\|_\infty\right)\|u\|_{H^{\eta,\eta/2}}.
	\end{align}
\end{lemma}

\begin{proof}
For $\theta\in\Theta$, consider the operator $L_\theta u := \frac{\theta}{2}\Delta_x u - \partial_t u$. 
For $\eta\in (0,1)$, by Theorem 6.2 in Chapter~4 of \cite{lionsVol2} (applied with $m=1$, $B_0u=u$ and $r=\eta/2$, so
that the compatibility conditions reduce to only the second condition in \eqref{eq:CR}, the first being empty because for the torus $\del\X=\varnothing$, and the third being empty as $\eta/2-1/2<0$. The compatibility conditions are thus clearly satisfied (by $\psi=0$) for the homogeneous boundary functions $u_0=0$), the operator $L_\theta$ is an isomorphism
of $H_{B,0}^{2+\eta, 1+\eta/2}(\Q)$ onto $H^{\eta, \eta/2}(Q)$.
It follows that there exists $C_\theta>0$ such that,
\begin{align*}
	\|\eL_{\theta, F}^{-1}u\|_{H^{2+\eta, 1+\eta/2}}&\leq C_\theta\|L_\theta(\eL_{\theta, F}^{-1}u)\|_{H^{\eta,\eta/2}}\\
	&= C_\theta \|- u + \Phi(F)\eL_{\theta, F}^{-1}u\|_{H^{\eta,\eta/2}}\\
	&\lle C_{\theta}(\|u\|_{H^{\eta,\eta/2}} + \|\Phi(F)\|_{C^\eta}\|\eL_{\theta,F}^{-1}u\|_{H^{\eta,\eta/2}}) \text{ (by \eqref{EqSobolevProductTwo})}\\
	&\lle C_\theta (\|u\|_{H^{\eta,\eta/2}} + \|\Phi(F)\|_{C^\eta}\|\eL_{\theta,F}^{-1}u\|_{H^{2,1}})\text{ (since $\eta<2$)}\\
	&\lle C_\theta (\|u\|_{H^{\eta,\eta/2}} + \|\Phi(F)\|_{C^\eta}(1+\|F\|_\infty)\|u\|_2)\\
	&\lle C_\theta (1+\|\Phi(F)\|_{C^\eta} + \|\Phi(F)\|_{C^\eta}\|F\|_{\infty})\|u\|_{H^{\eta,\eta/2}},
\end{align*}
where we used \eqref{eq: lipschitz2} in the second last inequality. The lemma follows if the map $\theta\mapsto C_\theta$ is bounded,
which by the assumed compactness of $\Theta$ certainly follows if it is continuous.
Now, $\|L_\theta v - L_\tau v\|_{H^{\eta,\eta/2}} = \|\frac{\theta-\tau}{2}\Delta_x v\|_{H^{\eta,\eta/2}}
	\lesssim |\theta-\tau| \|v\|_{H^{2+\eta, 1+\eta/2}}$, for any $v\in H_{B,0}^{2+\eta, 1+\eta/2}(\Q)$,
and \eqref{eq:refinedC} follows.

For the case $\eta\geq 1$, we note that Theorem 6.2 in Chapter 4 of \cite{lionsVol2} can still be applied to show $L_\theta:H_{B,0}^{2+\eta, 1+\eta/2}(\Q)\to H^{\eta,\eta/2}(\Q)$ is an isomorphism provided we can find a function $\psi\in H^{3/2+\eta, 1+\eta/2}(\Q)$ satisfying  the last two conditions in \eqref{eq:CR} for $u_0=0$. This basically consists of extending the traces at time $t=0$ to a function of space and time. This is always possible by Theorem 4.2 in Chapter 1 of \cite{lionsVol1}. The proof of the inequality for the case $\eta\leq 2$ is then identical to the case $\eta\in (0,1)$. If $\eta>2$, then it suffices to repeat the argument for $M-1$ more iterations. For completeness, we show how this is done below. We start from the second line of the previous display to bound $\|\eL_{\theta,F}\|_{H^{\eta+2,\eta/2+1}}$ by a constant multiple of
\begin{align*}
	&\| -u + \Phi(F)\eL_{\theta, F}^{-1}u\|_{H^{\eta,\eta/2}}\\
	&\lle \|u\|_{H^{\eta,\eta/2}} + \|\Phi(F)\|_{C^\eta}\|\eL_{\theta,F}^{-1}u\|_{H^{\eta,\eta/2}} \text{ (by \eqref{EqSobolevProductTwo})}\\
	&\lle \|u\|_{H^{\eta,\eta/2}} + \|\Phi(F)\|_{C^\eta}(\|u\|_{H^{\eta-2,\eta/2-1}} + \|\Phi(F)\|_{C^{\eta-2}}\|\eL_{\theta,F}^{-1}u\|_{H^{\eta-2,\eta/2-1}})\\
	&\lle (1+\|\Phi(F)\|_{C^\eta})\|u\|_{H^{\eta,\eta/2}} + \|\Phi(F)\|^2_{C^\eta}\|\eL_{\theta,F}^{-1}u\|_{H^{\eta-2,\eta/2-1}}\\
	&\cdots \\
	&\lle \left(\sum_{k=0}^{M-1}\|\Phi(F)\|_{C^{\eta}}^k\right) +\|\Phi(F)\|_{C^\eta}^M\|\eL_{\theta,F}^{-1}u\|_{H^{\eta-2(M-1), \eta/2 - (M-1)}}\\
	&\lle \left(\sum_{k=0}^{M-1}\|\Phi(F)\|_{C^{\eta}}^k\right) +\|\Phi(F)\|_{C^\eta}^M\|\eL_{\theta,F}^{-1}u\|_{H^{2,1}} \text{ (since }-2\leq\eta-2M\leq 0).
\end{align*}
We then obtain \eqref{eq:refinedCIterated} by applying \eqref{eq: lipschitz2}.

Finally, in the case that $F\in H^\eta(\X)$ with $\eta>d/2$ and $d\leq 3$, we note that $\eta-2 < \eta-d/2$, and so $H^{\eta}(\X)\subset C^{\eta-2}(\X)$ by Sobolev embedding. We then obtain \eqref{eq:refinedH} exactly as we did \eqref{eq:refinedCIterated}, except that in the second line of the last display above we apply \eqref{EqSobolevProductOne} instead of \eqref{EqSobolevProductTwo}.
\end{proof}

\begin{lemma}\label{lemma: boundedSolution}
	Let $\beta>2+d/2$. Assume that $u_0\in H^{1+\beta}(\X)$, and that $f\in C^\infty(\X)$ is such that $f\geq f_{\min}>0$. Then the following boundary value problem
	\begin{equation*}
		\left\{\begin{aligned}
			\eL_{\theta,f}u &= 0,\quad&& \text{on }\X\times (0,T],\\
			u(\cdot,0)& = u_0, && \text{on }\X,
		\end{aligned}\right.
	\end{equation*}
	has a unique strong solution $u_{\theta,f}\in H^{2+\beta, 1+\beta/2}(\Q)$. Furthermore, there exists a constant $C>0$ such that:
	\begin{align}
		\|u_{\theta, f}\|_{H^{2+\beta, 1+\beta/2}(\Q)}\le C\Bigl(1+\|f\|_{H^{\beta}(\X)}^{1+\beta/2}\Bigr)
	\end{align}
\end{lemma}
\begin{proof}
	Since $f$ is smooth, Theorem 5.3 in Chapter 4 of \cite{lionsVol2} directly implies the existence of a unique strong solution $u_{\theta,f}\in H^{2+\beta, 1+\beta/2}(\Q)$ to the PDE in the statement of the lemma (the existence of a compatibility function  satisfying \eqref{eq:CR} being guaranteed on the torus by the trace extension Theorem (Theorem 4.2 in Chapter 1 of \cite{lionsVol1}) as we have already seen in the previous lemma). As was seen in the proof of Lemma \ref{lemma: refinedLipschitz}, when $\X$ is the torus $\T^d$, the operator $L_\theta:H_{B,0}^{2+\beta, 1+\beta/2}(\Q)\to H^{\beta,\beta/2}(\Q)$ defined as $L_\theta = \frac{\theta}{2}\Delta_x -\del_t $ is an isomorphism. Furthermore, by Theorem 4.2 in Chapter 1 of \cite{lionsVol1}, one can also infer the existence of $\psi\in H^{2+\beta, 1+\beta/2}(\Q)$ (completely independent of $\theta$ and $f$) equal to $u_0$ at time $t=0$. It follows that $u_{\theta,f} - \psi\in H_{B,0}^{2+\beta,1+\beta/2}(\Q)$ and thus that
	\begin{align*}
		\|u_{\theta,f}\|_{H^{2+\beta,1+\beta/2}}&\leq \|u_{\theta,f}-\psi\|_{H^{2+\beta, 1+\beta/2}} + \|\psi\|_{H^{2+\beta, 1+\beta/2}}\\
		&\lle \|L_\theta (u_{\theta,f} -\psi)\|_{H^{\beta,\beta/2}} + \|\psi\|_{H^{2+\beta, 1+\beta/2}}\\
		&\leq \|L_\theta u_{\theta, f}\|_{H^{\beta, \beta/2}} +\|L_\theta \psi\|_{H^{\beta,\beta/2}} + \|\psi\|_{H^{2+\beta, 1+\beta/2}}\\
		&=\|fu_{\theta,f}\|_{H^{\beta,\beta/2}} + \Bigl( \|L_\theta\psi\|_{H^\beta,\beta/2} + \|\psi\|_{H^{2+\beta, 1+\beta/2}} \Bigr)\\
		&\lle \|f\|_{H^\beta}\|u_{\theta,f}\|_{H^{\beta,\beta/2}} + 1,
	\end{align*}
	where in the last inequality we used \eqref{EqSobolevProductOne} for bounding the first term, and Proposition 2.3 in Chapter 4 of \cite{lionsVol2} (applied to $\Delta_x$ and $\del_t$) for bounding the second term by $(\theta\vee 2)\|\psi\|_{H^{2+\beta, 1+\beta/2}}$, which can be bounded uniformly in $\theta$ since $\Theta$ is compact. By the interpolation inequality \eqref{EqInterpolation} applied with $\nu=2/(2+\beta)$, we further obtain the bound 
	\[\|u_{\theta,f}\|_{H^{2+\beta,1+\beta/2}} \lle 1 +\|f\|_{H^\beta} \|u_{\theta,f}\|_2^{2/(2+\beta)}\|u_{\theta,f}\|_{H^{2+\beta,1+\beta/2}}^{\beta/(2+\beta)}.\]
	By \eqref{eq: uniformBound}, we know that since $f$ is positive, $\|u_{\theta,f}\|_2\lle \|u_{\theta,f}\|_\infty\leq \|u_0\|_\infty$, and so we get
	\[\|u_{\theta,f}\|_{H^{2+\beta,1+\beta/2}} \lle 1 +\|f\|_{H^\beta}\|u_{\theta,f}\|_{H^{2+\beta,1+\beta/2}}^{\beta/(2+\beta)}.\]
	If $\|u_{\theta,f}\|_{H^{2+\beta, 1+\beta/2}}\leq 1$, then we can trivially bound the right hand side by $1+\|f\|_{H^\beta}^{1+\beta/2}$. In the case that $\|u_{\theta,f}\|_{H^{2+\beta, 1+\beta/2}}> 1$, we divide both sides of the inequality by $\|u\|_{H^{2+\beta, 1+\beta/2}}^{\beta/(2+\beta)}$ and conclude by subsequently exponentiating both sides to the power $1+\beta/2$.
	
\end{proof}

%%%%%%%%%%W%%%%%%%%%%%%%%%%%%%%%%%%%%%%%%%%%%%%
%% Single Appendix:                         %%
%%%%%%%%%%%%%%%%%%%%%%%%%%%%%%%%%%%%%%%%%%%%%%
%\begin{appendix}
%\section*{???}%% if no title is needed, leave empty \section*{}.
%\end{appendix}
%%%%%%%%%%%%%%%%%%%%%%%%%%%%%%%%%%%%%%%%%%%%%%
%% Multiple Appendixes:                     %%
%%%%%%%%%%%%%%%%%%%%%%%%%%%%%%%%%%%%%%%%%%%%%%
%\begin{appendix}
%\section{???}
%
%\section{???}
%
%\end{appendix}

%%%%%%%%%%%%%%%%%%%%%%%%%%%%%%%%%%%%%%%%%%%%%%
%% Support information, if any,             %%
%% should be provided in the                %%
%% Acknowledgements section.                %%
%%%%%%%%%%%%%%%%%%%%%%%%%%%%%%%%%%%%%%%%%%%%%%
\begin{acks}[Acknowledgments]
 The authors would like to thank Richard Nickl for insightful discussions. 
\end{acks}
\bibliographystyle{imsart-number} % Style BST file (imsart-number.bst or imsart-nameyear.bst)
\bibliography{bibliography}       % Bibliography file (usually '*.bib')

@incollection {AadStFlour,
    AUTHOR = {van der Vaart, Aad},
     TITLE = {Semiparametric statistics},
 BOOKTITLE = {Lectures on probability theory and statistics
              ({S}aint-{F}lour, 1999)},
    SERIES = {Lecture Notes in Math.},
    VOLUME = {1781},
     PAGES = {331--457},
 PUBLISHER = {Springer, Berlin},
      YEAR = {2002},
      ISBN = {3-540-43736-3},
   MRCLASS = {62Gxx},
  MRNUMBER = {1915446},
MRREVIEWER = {J.\ A.\ Melamed},
}

@article {vdV1991,
    AUTHOR = {van der Vaart, Aad},
     TITLE = {On differentiable functionals},
   JOURNAL = {Ann. Statist.},
  FJOURNAL = {The Annals of Statistics},
    VOLUME = {19},
      YEAR = {1991},
    NUMBER = {1},
     PAGES = {178--204},
      ISSN = {0090-5364,2168-8966},
   MRCLASS = {62G20 (62G05)},
  MRNUMBER = {1091845},
       DOI = {10.1214/aos/1176347976},
       URL = {https://doi.org/10.1214/aos/1176347976},
}

@book {Lunardi1995,
    AUTHOR = {Lunardi, Alessandra},
     TITLE = {Analytic semigroups and optimal regularity in parabolic
              problems},
    SERIES = {Modern Birkh\"auser Classics},
      NOTE = {[2013 reprint of the 1995 original] [MR1329547]},
 PUBLISHER = {Birkh\"auser/Springer Basel AG, Basel},
      YEAR = {1995},
     PAGES = {xviii+424},
      ISBN = {978-3-0348-0556-8; 978-3-0348-0557-5},
   MRCLASS = {47D06 (01A75 34G20 35Kxx 46M35 46N20 47N20 58D25)},
  MRNUMBER = {3012216},
}

@book{krylov,
	title={Lectures on Elliptic and Parabolic Equations in Sobolev Spaces},
	author={Krylov, N.V.},
	isbn={9780821846841},
	lccn={2008016051},
	series={Graduate Studies in Mathematics, Graduate Studies in Mathema},
	url={https://books.google.nl/books?id=8RzqBwAAQBAJ},
	year={2008},
	publisher={American Mathematical Society}
}

@book{lionsVol1,
	title={Non-homogeneous boundary value problems and applications: Vol. 1},
	author={Lions, Jacques Louis and Magenes, Enrico},
	volume={181},
	year={2012},
	publisher={Springer Science \& Business Media}
}

@book{lionsVol2,
	title={Non-homogeneous boundary value problems and applications: Vol. 2},
	author={Lions, Jacques Louis and Magenes, Enrico},
	volume={181},
	year={2012},
	publisher={Springer Science \& Business Media}
}

@article{kekkonen,
	title = {Consistency of Bayesian inference with Gaussian process priors for a parabolic inverse problem},
	author = {Hanne Kekkonen},
	year = {2022},
	doi = {10.1088/1361-6420/ac4839},
	language = {English},
	volume = {38},
	journal = {Inverse Problems},
	issn = {0266-5611},
	publisher = {IOP Publishing},
	number = {3},
}

@article{giordano,
	title={Consistency of Bayesian inference with Gaussian process priors in an elliptic inverse problem},
	author={Giordano, Matteo and Nickl, Richard},
	journal={Inverse problems},
	volume={36},
	number={8},
	pages={085001},
	year={2020},
	publisher={IOP Publishing}
}

@article{nicklVdG,
	title={Convergence rates for penalized least squares estimators in PDE constrained regression problems},
	author={Nickl, Richard and van de Geer, Sara and Wang, Sven},
	journal={SIAM/ASA Journal on Uncertainty Quantification},
	volume={8},
	number={1},
	pages={374--413},
	year={2020},
	publisher={SIAM}
}

@book{GhosalVdV,
	author = {Ghosal, Subhashis and Vaart, Aad},
	year = {2017},
	month = {06},
	pages = {1-646},
	title = {Fundamentals of Nonparametric Bayesian Inference},
	isbn = {9780521878265},
	journal = {Fundamentals of Nonparametric Bayesian Inference},
	doi = {10.1017/9781139029834}
}

@article {vdVvZ2008,
    AUTHOR = {van der Vaart, A. W. and van Zanten, J. H.},
     TITLE = {Rates of contraction of posterior distributions based on
              {G}aussian process priors},
   JOURNAL = {Ann. Statist.},
  FJOURNAL = {The Annals of Statistics},
    VOLUME = {36},
      YEAR = {2008},
    NUMBER = {3},
     PAGES = {1435--1463},
      ISSN = {0090-5364,2168-8966},
   MRCLASS = {62G05 (60G15)},
  MRNUMBER = {2418663},
MRREVIEWER = {Theofanis\ Sapatinas},
       DOI = {10.1214/009053607000000613},
       URL = {https://doi.org/10.1214/009053607000000613},
}

@article {vdVvZ2007,
    AUTHOR = {van der Vaart, Aad and van Zanten, Harry},
     TITLE = {Bayesian inference with rescaled {G}aussian process priors},
   JOURNAL = {Electron. J. Stat.},
  FJOURNAL = {Electronic Journal of Statistics},
    VOLUME = {1},
      YEAR = {2007},
     PAGES = {433--448},
      ISSN = {1935-7524},
   MRCLASS = {62G05 (60G15 62C10)},
  MRNUMBER = {2357712},
MRREVIEWER = {J.\ A.\ Melamed},
       DOI = {10.1214/07-EJS098},
       URL = {https://doi.org/10.1214/07-EJS098},
}

@book{vaartghosal_errata,
  title={Errata to Fundamentals of nonparametric Bayesian inference},
  author={Ghosal, Subhashis and Van der Vaart, Aad},
  year={2017},
  publisher={https://diamhomes.ewi.tudelft.nl/$\sim$avandervaart/books/bayes/errata.pdf}
}

@article {vaartghosal2007,
    AUTHOR = {Ghosal, Subhashis and van der Vaart, Aad},
     TITLE = {Convergence rates of posterior distributions for non-i.i.d.
              observations},
   JOURNAL = {Ann. Statist.},
  FJOURNAL = {The Annals of Statistics},
    VOLUME = {35},
      YEAR = {2007},
    NUMBER = {1},
     PAGES = {192--223},
      ISSN = {0090-5364,2168-8966},
   MRCLASS = {62G20 (62G08)},
  MRNUMBER = {2332274},
MRREVIEWER = {Stergios\ B.\ Fotopoulos},
       DOI = {10.1214/009053606000001172},
       URL = {https://doi.org/10.1214/009053606000001172},
}

@book {AadsBoek,
	AUTHOR = {van der Vaart, A. W.},
	TITLE = {Asymptotic statistics},
	SERIES = {Cambridge Series in Statistical and Probabilistic Mathematics},
	VOLUME = {3},
	PUBLISHER = {Cambridge University Press, Cambridge},
	YEAR = {1998},
	PAGES = {xvi+443},
	ISBN = {0-521-49603-9; 0-521-78450-6},
	MRCLASS = {62-02 (62E20 62F05 62F12 62G07 62G09 62G20)},
	MRNUMBER = {1652247},
	MRREVIEWER = {Nancy Reid},
	DOI = {10.1017/CBO9780511802256},
	URL = {https://doi.org/10.1017/CBO9780511802256},
}

@article{castillo,
  title={A semiparametric Bernstein--von Mises theorem for Gaussian process priors},
  author={Castillo, Isma{\"e}l},
  journal={Probability Theory and Related Fields},
  volume={152},
  pages={53--99},
  year={2012},
  publisher={Springer}
}

@article{monard2021,
  title={Consistent Inversion of Noisy Non-Abelian X-Ray Transforms},
  author={Monard, Fran{\c{c}}ois and Nickl, Richard and Paternain, Gabriel P},
  journal={Communications on Pure and Applied Mathematics},
  volume={74},
  number={5},
  pages={1045--1099},
  year={2021},
  publisher={Wiley Online Library}
}

@book{Nickl23,
	author = {R. Nickl},
year={2023},
title={Bayesian Non-linear Inverse Problems},
series={Z\"urich Lectures in Advanced Mathematics},
publisher={ETH, Z\"urich},
}

@article {Dudley73,
    AUTHOR = {Dudley, R. M.},
     TITLE = {Sample functions of the {G}aussian process},
   JOURNAL = {Ann. Probability},
  FJOURNAL = {The Annals of Probability},
    VOLUME = {1},
      YEAR = {1973},
    NUMBER = {1},
     PAGES = {66--103},
      ISSN = {0091-1798},
   MRCLASS = {60G15 (60G17)},
  MRNUMBER = {346884},
       DOI = {10.1214/aop/1176997026},
       URL = {https://doi.org/10.1214/aop/1176997026},
}

@book {vdV_wellner2023,
    AUTHOR = {van der Vaart, A. W. and Wellner, Jon A.},
     TITLE = {Weak convergence and empirical processes---with applications
              to statistics},
    SERIES = {Springer Series in Statistics},
   EDITION = {Second},
 PUBLISHER = {Springer, Cham},
      YEAR = {2023},
     PAGES = {xvii+679},
      ISBN = {978-3-031-29038-1; 978-3-031-29040-4},
   MRCLASS = {60-02 (60B12 60F05 62G30)},
  MRNUMBER = {4628026},
       DOI = {10.1007/978-3-031-29040-4},
       URL = {https://doi.org/10.1007/978-3-031-29040-4},
}

@article{magra,
author = {Adel Magra and Aad van der Vaart and Harry van Zanten},
title = {{Semi-parametric Bernstein-von Mises theorem in linear inverse problems}},
volume = {19},
journal = {Electronic Journal of Statistics},
number = {1},
publisher = {Institute of Mathematical Statistics and Bernoulli Society},
pages = {1855 -- 1888},
year = {2025},
doi = {10.1214/25-EJS2372},
URL = {https://doi.org/10.1214/25-EJS2372}
}

@article {RayvdV,
    AUTHOR = {Ray, Kolyan and van der Vaart, Aad},
     TITLE = {Semiparametric {B}ayesian causal inference},
   JOURNAL = {Ann. Statist.},
  FJOURNAL = {The Annals of Statistics},
    VOLUME = {48},
      YEAR = {2020},
    NUMBER = {5},
     PAGES = {2999--3020},
      ISSN = {0090-5364,2168-8966},
   MRCLASS = {62D20 (62D10 62G08 62G15 62G20)},
  MRNUMBER = {4152632},
       DOI = {10.1214/19-AOS1919},
       URL = {https://doi.org/10.1214/19-AOS1919},
}

@article {GiordanoKekkonen,
    AUTHOR = {Giordano, Matteo and Kekkonen, Hanne},
     TITLE = {Bernstein--von {M}ises theorems and uncertainty quantification
              for linear inverse problems},
   JOURNAL = {SIAM/ASA J. Uncertain. Quantif.},
  FJOURNAL = {SIAM/ASA Journal on Uncertainty Quantification},
    VOLUME = {8},
      YEAR = {2020},
    NUMBER = {1},
     PAGES = {342--373},
      ISSN = {2166-2525},
   MRCLASS = {62G20 (62F15 65N21)},
  MRNUMBER = {4069334},
MRREVIEWER = {Guang-Hui\ Zheng},
       DOI = {10.1137/18M1226269},
       URL = {https://doi.org/10.1137/18M1226269},
}

@article {Knapiketal2011,
    AUTHOR = {Knapik, B. T. and van der Vaart, A. W. and van Zanten, J. H.},
     TITLE = {Bayesian inverse problems with {G}aussian priors},
   JOURNAL = {Ann. Statist.},
  FJOURNAL = {The Annals of Statistics},
    VOLUME = {39},
      YEAR = {2011},
    NUMBER = {5},
     PAGES = {2626--2657},
      ISSN = {0090-5364,2168-8966},
   MRCLASS = {62G05 (62F15 62G20)},
  MRNUMBER = {2906881},
MRREVIEWER = {Kaushik\ Ghosh},
       DOI = {10.1214/11-AOS920},
       URL = {https://doi.org/10.1214/11-AOS920},
}

@article{Nickl2018,
	author = {R. Nickl},
	title = {Bernstein-von Mises Theorems for statistical inverse problems I: Schr\"odinger Equation},
	year = {2018},
	journal = {Journal of the European Mathematical Society},
	volume = {22},
	number = {8},
	pages = {2697-2750},
	DOI = {https://doi.org/10.4171/JEMS/975},
}

@article {ismaeljudith,
    AUTHOR = {Castillo, Isma\"{e}l and Rousseau, Judith},
     TITLE = {A {B}ernstein--von {M}ises theorem for smooth functionals in
              semiparametric models},
   JOURNAL = {Ann. Statist.},
  FJOURNAL = {The Annals of Statistics},
    VOLUME = {43},
      YEAR = {2015},
    NUMBER = {6},
     PAGES = {2353--2383},
      ISSN = {0090-5364},
   MRCLASS = {62G20 (62F15 62M05)},
  MRNUMBER = {3405597},
MRREVIEWER = {Shibin Zhang},
       DOI = {10.1214/15-AOS1336},
       URL = {https://doi.org/10.1214/15-AOS1336},
}

@article {MonardNicklPternain2021,
    AUTHOR = {Monard, Fran\c cois and Nickl, Richard and Paternain, Gabriel
              P.},
     TITLE = {Statistical guarantees for {B}ayesian uncertainty
              quantification in nonlinear inverse problems with {G}aussian
              process priors},
   JOURNAL = {Ann. Statist.},
  FJOURNAL = {The Annals of Statistics},
    VOLUME = {49},
      YEAR = {2021},
    NUMBER = {6},
     PAGES = {3255--3298},
      ISSN = {0090-5364,2168-8966},
   MRCLASS = {62F15 (65N21)},
  MRNUMBER = {4352530},
       DOI = {10.1214/21-aos2082},
       URL = {https://doi.org/10.1214/21-aos2082},
}

@article {Knapiketal2013,
    AUTHOR = {Knapik, B. T. and van der Vaart, A. W. and van Zanten, J. H.},
     TITLE = {Bayesian recovery of the initial condition for the heat
              equation},
   JOURNAL = {Comm. Statist. Theory Methods},
  FJOURNAL = {Communications in Statistics. Theory and Methods},
    VOLUME = {42},
      YEAR = {2013},
    NUMBER = {7},
     PAGES = {1294--1313},
      ISSN = {0361-0926},
   MRCLASS = {62G05 (35K15 35R50 62C20 62F15 62G15 62G20)},
  MRNUMBER = {3031282},
MRREVIEWER = {Kaushik Ghosh},
       DOI = {10.1080/03610926.2012.681417},
       URL = {https://doi.org/10.1080/03610926.2012.681417},
}

@article{Ray2013,
author = {Kolyan Ray},
title = {{Bayesian inverse problems with non-conjugate priors}},
volume = {7},
journal = {Electronic Journal of Statistics},
number = {none},
publisher = {Institute of Mathematical Statistics and Bernoulli Society},
pages = {2516 -- 2549},
year = {2013},
doi = {10.1214/13-EJS851},
URL = {https://doi.org/10.1214/13-EJS851}
}

@article{Yan2020,
author = {Shota Gugushvili and Aad van der Vaart and Dong Yan},
title = {{Bayesian linear inverse problems in regularity scales}},
volume = {56},
journal = {Annales de l'Institut Henri Poincaré, Probabilités et Statistiques},
number = {3},
publisher = {Institut Henri Poincaré},
pages = {2081 -- 2107},
year = {2020},
doi = {10.1214/19-AIHP1029},
URL = {https://doi.org/10.1214/19-AIHP1029}
}

@article{SzabovdV2024,
author = {Geerten Koers and Botond Szabó and Aad van der Vaart},
title = {Linear methods for nonlinear inverse problems},
year = {2024},
journal = {preprint},
}

@article{brezis1983,
  title={Analyse fonctionnelle},
  author={Brezis, Ha{\"\i}m},
  journal={Th{\'e}orie et applications},
  year={1983},
  publisher={masson}
}

@misc{nickl2025,
      title={Bernstein-von Mises theorems for time evolution equations}, 
      author={Richard Nickl},
      year={2025},
      eprint={2407.14781},
      archivePrefix={arXiv},
      primaryClass={math.ST},
      url={https://arxiv.org/abs/2407.14781}, 
}

@book{EvansPDE,
  author    = {Evans, Lawrence C.},
  title     = {Partial Differential Equations},
  series    = {Graduate Studies in Mathematics},
  volume    = {19},
  edition   = {2},
  publisher = {American Mathematical Society},
  year      = {2010},
}

%% or include bibliography directly:
% \begin{thebibliography}{}
% \bibitem{b1}
% \end{thebibliography}

\end{document}